\documentclass{amsart}
\usepackage{mathtools}  
\usepackage{mathrsfs}   
\usepackage{bm}         
\usepackage{physics}
\usepackage{tikz-cd}
\usepackage{enumitem}
\usepackage{amssymb}
\usepackage{stmaryrd}
\usepackage[T1]{fontenc}
\usepackage[utf8]{inputenc} 
\usepackage{CJKutf8}        
\usepackage{mathtools}  
\usepackage{extarrows}
\usepackage{mathrsfs}
\usepackage[mathcal]{eucal}
\usepackage{footnote}
\usepackage[backend=biber, style=alphabetic]{biblatex}
\usepackage{graphicx}
\usepackage{caption}
\usepackage{bbm}

\usepackage[dvipsnames,svgnames,table]{xcolor}
\usepackage[unicode=true,pdfusetitle,
 bookmarks=true,bookmarksnumbered=false,bookmarksopen=false,
 breaklinks=false,pdfborder={0 0 0},pdfborderstyle={},backref=false,colorlinks=true]{hyperref}
 
\hypersetup{
    colorlinks, linkcolor=OliveGreen,
    citecolor=OliveGreen, urlcolor=OliveGreen
}
\usepackage[nameinlink,capitalise,noabbrev]{cleveref}

\crefalias{theorem}{theorem}
\crefalias{proposition}{proposition}
\crefalias{corollary}{corollary}
\crefalias{lemma}{lemma}
\crefalias{example}{example}
\crefalias{remark}{remark}

\theoremstyle{definition}
\newtheorem{definition}{Definition}[section]
\newtheorem{example}[definition]{Example}

\newtheorem{remark}[definition]{Remark}

\theoremstyle{plain}
\newtheorem{theorem}[definition]{Theorem}
\newtheorem{proposition}[definition]{Proposition}
\newtheorem{corollary}[definition]{Corollary}

\newtheorem{lemma}[definition]{Lemma}

\newcommand{\adjoint}{\ \substack{ \longrightarrow \\[-1em] \longleftarrow}\ }

\title[Cohomology in Supergeometry]{Derived Methods in Supergeometry: Fundamental Classes and Cohomology}
\author{Marcel Dang}
\begin{document}
\begin{titlepage}
    \begin{abstract}
    In this paper we apply ideas and methods from derived algebraic geometry to supergeometry. Using this, we define virtual fundamental classes in supergeometry and obtain the $\Theta$-classes of the moduli of curves as the pushforward of a super fundamental class. \\
    We also extend Simpson's transmutation stacks, which are geometric avatars of Betti, de Rham and Dolbeault cohomology.
    Leveraging this more modern viewpoint, we obtain new proofs for Penkov's theorem on $D$-Modules and the invariance of de Rham cohomology under fermionic thickenings. \\
    The goal of this paper is to demonstrate the usage of ideas and methods coming from derived algebraic geometry in the supergeometric setting as derived and super have geometric similarities, which lead to the same considerations when adapting classical notions to their respective settings.
\end{abstract}
\maketitle
\begin{figure}[h]
    \centering
    \includegraphics[width=0.8\linewidth]{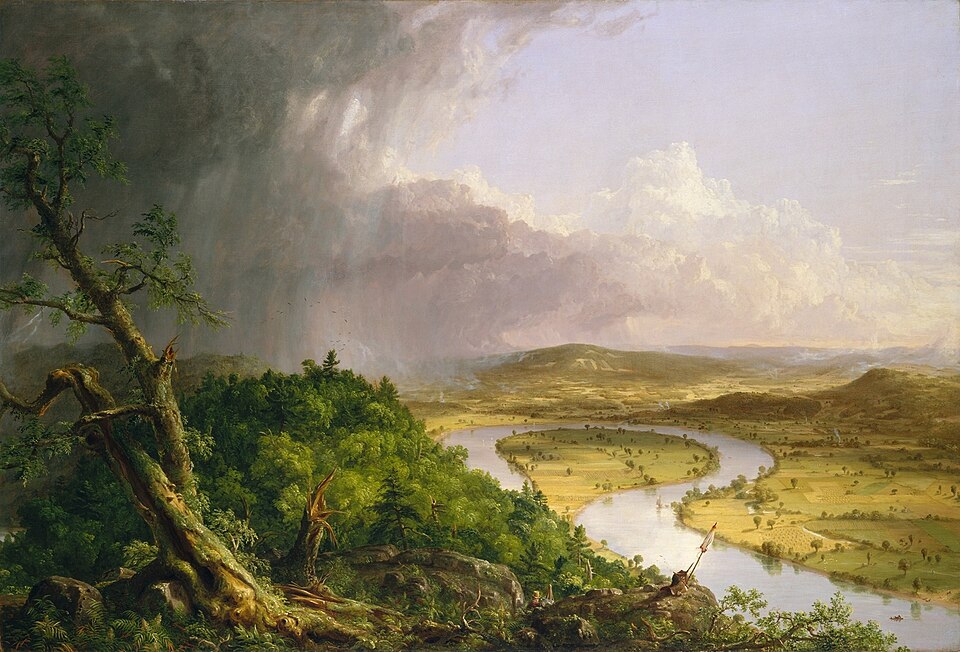}
    \caption*{Cole, Thomas (1863). The Oxbow via Wikimedia Commons}
\end{figure}
\end{titlepage}

\tableofcontents
\section{Introduction}
\subsection{Background}
The category of schemes is built out of the category of commutative rings as its local model. By realizing that all one needed were commutative algebra objects in any symmetric monoidal categoy, one may contemplate variants of algebraic geometry by formally declaring $\mathsf{CAlg}(\mathcal{C})^{\mathsf{op}} = \mathsf{Aff}_\mathcal{C}$ to be the local model, which we built our algebraic geometry on top of.

Derived algebraic geometry (DAG) is a variant of algebraic geometry emerging out of the category of simplicial abelian groups $\mathcal{C}=\mathsf{Ab}_\Delta$, considered up to weak equivalence after taking commutative algebra objects, which enlarges the category of commutative rings by higher nilpotent thickenings and, in turn, forms a more well-behaved category with many desirable properties. For instance, a technical upshot is that if one considers a base-change situation, it will be true without flatness assumption and it turns the derived category into the intrinsic notion of module category, which naturally equips us with a 6-functor formalism and lets us realize the tangent complex of a non-regular scheme as a natural geometric object in the category of derived schemes.

The enlargement of the category of schemes also yields conceptual power as it explains excess intersection situations, which are controlled by Serre's intersection formula, as the intrinsic notion of intersections for derived schemes. Similarly, virtual fundamental classes, originating from Kontsevichs hidden smoothness principle and  constructed by Behrend-Fantechi, allowed one to compute the expected intersection numbers in Gromov-Witten theory, and have a natural interpretation using DAG. The virtual fundamental class is the fundamental class of a natural derived enhancement of the moduli stack that describes the enumerative problem one is interested in. However, the conceptual clarity does not come for free, as it is difficult to construct derived enhancements of moduli stacks and extract the virtual fundamental classes outside of the quasi-smooth situation of Behrend-Fantechi, in which case they are given by the following formula:
\begin{theorem}[cf. {\cite[Proposition 5.6.]{behrend1997intrinsic}}]
    Let $X$ be a smooth Deligne-Mumford stack over $\mathbb{C}$ and $E^\bullet$ a perfect, locally free obstruction theory then the virtual fundamental class is 
    \begin{align}
        [X,E^\bullet] = c_{\mathrm{top}}(H^1(E^\bullet)^\vee) \cap [X].
    \end{align}
\end{theorem}

On the other hand, there is another natural variant of $\mathcal{C}$-algebraic geometry, designed for the study of supersymmetric phenomena in mathematical physics, called supergeometry, where $\mathcal{C}=\mathsf{Ab}_{\mathbb{Z}_2}$, which is the category of $\mathbb{Z}_2$-graded abelian groups with the Koszul sign rule as symmetric monoidal structure. In this setting one enlarges the category of rings by so-called fermionic nilpoent thickenings, which allows one to study bosonic and fermionic variables on even grounds and allows one to make supersymmetry manifest, meaning that one can realize supersymmetry as a supergroup symmetry of the physical system.

As both variants are nilpotent thickenings of ordinary algebraic geometry they share similar geometric interpretations, which lead to similar considerations and phenomena when studying either of them. 
In the study of the moduli stack of super Riemann surfaces $\mathfrak{M}_{g,n}$, it was proven in \cite{stanford2020jtgravityensemblesrandom} and \cite{norbury2020enumerative} that the $\Theta$-classes introduced in \cite{norbury2023new} are related to volumes of the moduli stack of super Riemann surfaces, where the symplectic volume of $\mathfrak{M}_{g,n}$ is shown to be given by the formula:
\begin{align}
    \mathsf{Vol}(\mathfrak{M}_{g,n}) \coloneqq \int_{\mathcal{M}_{g,n}} e(\mathcal{E}^\vee_{g,n}) \mathrm{exp(\omega_{\mathrm{WP}})}.
\end{align}
The purpose of this paper is to apply methods of DAG to supergeometry.
\subsection{Fundamental Classes}
By observing that the virtual fundamental class is the intrinsic notion of fundamental class for a derived scheme or stack, we prove a d\'evissage theorem, which is inspired by its derived analogue.
\begin{theorem}[D\'evissage]
    Let $X$ be a noetherian geometric superstack then the map 
    \begin{align}
        j_*:K(\mathsf{sCoh}(X_\mathsf{bos})) \longrightarrow K(\mathsf{sCoh}(X)).
    \end{align}
    is an equivalence.
\end{theorem}
This allows us to define the notion of a $K$-theoretic super fundamental class.
\begin{definition}[Super Fundamental $K$-Class]
    Let $X$ be an Artin superstack and $\iota: X_\mathsf{bos}\rightarrow X$ be the inclusion of the bosonic truncation. We define the \textit{super fundamental class} $ [X]$ to be the unique class that satisfies
    \begin{align}
        \iota_*[X]^K = [\mathcal{O}_X] \in K_0(\mathsf{sCoh}(X_\mathsf{bos}))
    \end{align}
\end{definition}
Furthermore, we prove that, after a Grothendieck-Riemann-Roch transformation, the super fundamental class bears strong similarity to the quasi-smooth situation of DAG:

\begin{proposition}
    Let $X$ be a smooth Deligne-Mumford superstack and let $[X] \in A_*(X_\mathsf{bos})$ its super fundamental class then 
    \begin{align}
        [X] \simeq c_r(\mathcal{E}^\vee) \cdot [X]. 
    \end{align}
\end{proposition}
As an application, we prove that the $\Theta$-classes in \cite{norbury2023new} on the moduli stack of curves arise by pushing forward the super fundamental class of the moduli stack of super Riemann surfaces $\mathfrak{M}_{g,n}$ onto the moduli of curves. This provides an interpretation of $\mathfrak{M}_{g,n}$ as a natural super enhancement of the moduli of $r$-spin curves for $r=2$, and the volume calculations as intersection theory on the super enhancement. In particular, this demonstrates that some enumerative problems are more naturally described through their super enhancement, which is 1-categorical in nature as opposed to its derived counterpart. 
\begin{proposition}
Let $p: \overline{\mathcal{SM}}_{g,n}\longrightarrow \overline{\mathcal{M}}_{g,n}$ be the forgetful morphism, forgetting the spin structure, then we have the following identity:
\begin{align}
	p_*[\overline{\mathfrak{M}}_{g,n}] = \Theta_{g,n} \cap [\overline{\mathcal{M}}_{g,n}].
\end{align}
\end{proposition}

\subsection{Cohomology}
The $\infty$-categorical nature of DAG forces one to have robust and flexible definitions, which isolate a certain core aspect of the original definition, which we will use to obtain definitions in the supergeometric setting, where in particular we are interested in applying it to de Rham cohomology.

Classically, it is known that standard cohomological invariants of supermanifolds, such as de Rham or singular cohomology, are isomorphic to those of their underlying reduced manifolds. In this sense, the superstructure cannot be detected cohomologically. However, we can still analyze these invariants, by shifting our perspective from the cohomology groups themselves to the geometric spaces that represent them. This was originally introduced by Simpson \cite{simpson1999algebraic} and extended by Drinfeld \cite{drinfeld2024prismatization} and upcoming work of Bhatt-Lurie \cite{bhatt2022prismatic} \cite{BhattLurieGauges}. We extend Simpson's classical transmutations\footnote{This was coined by Bhatt in 
\cite{bhatt2022prismatic}} to the supergeometric setting, explicitly constructing the Betti, de Rham, and Dolbeault stacks of superschemes. While our approach to define the de Rham stack is based on geometric intuition, the de Rham stack has also been defined using a different categorical approach in Carchedi \cite{carchedi2025quasicoherentsheavesdmodulesderived}, to develop a theory of differential equations on derived  differential superstacks. Nevertheless, both approaches, when applied to underived superschemes, produce the same functor of points. In the case of the de Rham transmutation we observe that the de Rham stack of a superscheme and its underlying bosonic truncation are identical. From this perspective, the following classical theorems from supergeometry are the same theorem. First, this identity yields a new and conceptually transparent proof for Penkov's theorem \cite{penkov1983d}, allowing us to extend it to pre-superstacks:
\begin{theorem}[Penkov for Pre-Superstacks]
    Let $X$ be a pre-superstack, then the category of crystals on $X$ is equivalent to the category of crystals on $X_{\mathsf{bos}}$. 
\end{theorem}
Similarly, we give a new proof for de Rham cohomology and de Rham supercohomology being isomorphic
\begin{theorem}
    Let $X$ be a smooth superscheme then 
    \begin{align}
      R^\bullet\Gamma(X_{\mathsf{dR}}, \mathcal{O}_{X_{\mathsf{dR}}}) \simeq H_{\mathsf{dR}}(X, \mathbb{C})
    \end{align}
\end{theorem}
and we also give a simplified proof for the Beilinson-Bernstein localization theorem for quasi-reductive algebraic supergroups, originally due to Penkov \cite{penkov1994generic}.
\begin{theorem}[{\cite[Theorem 1.3.]{penkov1994generic}}]
    Let $G$ be a quasireductive algebraic supergroup over $\mathbb{C}$ and $B$ a Borel subgroup, then
    \begin{align}
        D\mathsf{Mod}(G/B) \simeq (\mathcal{U}(\mathfrak{g}_0)/\mathsf{ker}(\chi))\mathsf{-Mod}.
    \end{align}
\end{theorem}
In the Betti setting, we give a refinement of the Betti stack, using the analogoue of the monodromy equivalence, called the exodromy equivalence, where one replaces the homotopy type of a space by the so-called exit path category. This allows us to extend the Betti stack to the setting of $P$-constructible sheaves for a poset $P$. 
\begin{theorem}
    Let $X \to P$ be a stratified space such that its exit path category $\mathsf{Exit}(X,P)$ is an $\infty$-category then   
    \begin{align}
        \mathsf{QCoh}(X_{B_P} \times \mathsf{Spec}(A)) \simeq \mathsf{Fun}(\mathsf{Exit}(X,P), \mathsf{Mod}_A).
    \end{align}
\end{theorem}

In the Dolbeault setting, we prove the Dolbeault stack is the only classical transmutation that inherently captures the underlying supergeometric data, which we can capture in the following theorem.

\begin{theorem}
    Let $X$ be a smooth scheme of purely odd dimension $m$, then 
    \begin{align}
        X^{\mathsf{nil}}_{\mathsf{Dol}} \simeq X_{\mathsf{Dol}}.
    \end{align}
\end{theorem}
This means that the odd Higgs fields, one obtains by going to the super setting will be nilpotent. 
\subsection{Quasi-coherent Sheaves}
 To prove that quasi-coherent sheaves on the de Rham, stratified Betti or Dolbeault stack can be identified with $D$-modules, $P$-constructible sheaves or Higgs sheaves one needs several technical theorems such as descent for $\mathsf{QCoh}$, a base change theorem for pre-superstacks and an open-closed decomposition for superschemes also known as recollement, which are of independent interest. Instead of verifying these theorems directly in the super setting, we will deduce them from another variant of $\mathcal{C}$-algebraic geometry, where $\mathcal{C}=\mathsf{Ab}_{\mathbb{Z}}$, i.e. the category of $\mathbb{Z}$-graded abelian groups with the Koszul sign in the symmetric monoidal structure, which is also known as Dirac geometry. We will show that the category of quasi-coherent Dirac sheaves of a superstack is equivalent to the category of quasi-coherent supersheaves:
\begin{theorem}
    Let $X$ be a pre-superstack viewed as a Dirac prestack. Then there exists an equivalence of symmetric monoidal categories 
    \begin{align}
        \mathsf{QCoh}_{\mathsf{super}}(X) \simeq \mathsf{QCoh}_{\mathsf{Dirac}}(X).
    \end{align}
\end{theorem}
Using this equivalence descent, base-change and recollement follow from the Dirac geometric case.
Moreover, we prove that the category of stacks is a colocalization of superstacks:
\begin{theorem}[Colocalization]
    The functor $\iota$ is fully faithful and admits a right adjoint given by $(-)_{\mathsf{bos}}$. 
\end{theorem}
This proves \cite[Remark 2.28.]{bruzzo2025foundations} for sheaves of anima.
\subsection{Organization}
The paper is structured as follows.  Section 2 develops the theory of superstacks and their derived categories in the fpqc-topology from an $\infty$-categorical perspective, where we establish the tools required for the theorems in section 4, using a reduction argument to deduce them from the Dirac geometric setting. Section 3 introduces the transmutation stacks of Simpson. In the de Rham setting we adapt definitions from derived algebraic geometry in a straightforward way and then demonstrate that we recover the classical definitions. Furthermore, we give new geometric proofs of classical theorems in supergeometry. In the Betti setting, we refine the notion of Betti stack to also allow for $P$-stratified constructible sheaves. In the Dolbeault setting, we show that the Dolbeault stack is not agnostic to supergeometric information. In section 4, we will study the $K$-theory spectrum of the category of coherent sheaves on a noetherian superstack, which will allow us to prove the aforementioned d\'evissage theorem and identify $\Theta$-classes as the fundamental class of the moduli stack of super Riemann surfaces.
\subsection{Terminology and Convention}
 \begin{itemize}
     \item By category we will mean an $(\infty,1)$-category. By $1$-category we will mean a $(1,1)$-category.
     \item We denote the category of $\infty$-groupoids/homotopy types/spaces by $\mathsf{Ani}$ and call an object an anima.
     \item By animation, we mean the sifted cocompletion of a category in the $\infty$-categorical sense.
     \item For an abelian group $G$ we denote the category of $G$-graded abelian groups by $\mathsf{Ab}_G\coloneqq \mathsf{Fun}(G,\mathsf{Ab})$ and we use the Day convolution as symmetric monoidal structure with the Koszul sign in the braiding isomorphism.
     \item By a superring (resp. Dirac ring) we mean a commutative algebra object in $\mathbb{Z}_2$-graded (resp. $\mathbb{Z}$-graded) abelian groups, and we denote its category by $\mathsf{sCalg}\coloneqq \mathsf{CAlg}(\mathsf{Ab}_{\mathbb{Z}_2})$ (resp. $\mathsf{DCalg}\coloneqq \mathsf{CAlg}(\mathsf
     {Ab}_\mathbb{Z})$).
     \item By affine superscheme (resp. Dirac scheme) we mean an object in the opposite category of superrings (resp. Dirac rings), which we denote by $\mathsf{sAff}$ (resp. $\mathsf{DAff}$).
     \item  We call the unbounded derived category on an algebraic stack $X$ the category of quasi-coherent sheaves on $X$ and denote it by $\mathsf{QCoh}(X)$.
     \item To avoid clutter, we will drop extra decorations when no confusion can arise, if not, we will write $?\mathsf{QCoh}(X)$ and $?\mathsf{Aff}$ for $? \in \{\mathsf{D},\mathsf{s},\mathsf{d}\}$ which stands for Dirac, super or derived respectively.
     \item By stack we mean an accessible sheaf of anima in the fpqc topology, and we denote the category of stacks as $?\mathsf{Stk} \coloneqq \mathsf{Shv}^{\mathrm{acc}}_{\mathsf{fpqc}}(?\mathsf{Aff})$
 \end{itemize}  

\subsection{Acknowledgement}
First, I would like to thank my supervisor Paul Norbury for his guidance throughout this project and his numerous suggestions regarding an earlier versions of this work. I owe special thanks to Gustav Berth and Lior Yanovski, who have taught me everything I know about $\infty$-categories. I am grateful to Oliver Li and Fei Peng for numerous enlightening discussions about algebraic stacks. I would also like to thank Tianqi Feng, Peter Haine and Max Treutlein for useful discussion. Lastly, I would like to thank Peter Scholze for a very helpful answer to my email regarding $!$-functors. The author is supported by the University of Melbourne Graduate Research Scholarship, and the Science Abroad Travelling Scholarship offered by the University of Melbourne.

We used LLMs to improve the writing. All mathematical errors are our own.
\section{Superstacks}
 In this section we develop the theory of derived categories on superstacks. We achieve this by reducing the constructions from \cite[Section 3]{hesselholt2024dirac} from the $\mathbb{Z}$-graded case to the $\mathbb{Z}_2$-graded case. We will prove that a superstack is a Dirac stack and we will prove that for superstacks the category of quasi-coherent Dirac sheaves will reduce to the category of quasi-coherent supersheaves. Consequently, we obtain the necessary theory in the $\mathbb{Z}_2$-graded case without having to give the analogous arguments.
 \subsection{Recollections from Dirac Geometry}
Let us now state the theorems that we adopt to the super setting. Let us start by defining the category of quasi-coherent sheaves.
   \begin{definition}
    If $S \simeq  \mathsf{Spec}(A)$ is an affine Dirac stack, then 
    \begin{align}
        \mathsf{QCoh}(S)_{\geq0} \simeq  \mathsf{An}(\mathsf{Mod}(A)^\heartsuit)
    \end{align}
    is the animation of the abelian category of $A$-modules and 
    \begin{align}
         \mathsf{QCoh}(S)_{\geq0} \longrightarrow \mathsf{QCoh}(S) \simeq  \mathsf{Sp}(\mathsf{An}(\mathsf{Mod}(A)^\heartsuit)).
    \end{align}
\end{definition}
 
\begin{definition}
   The category of quasi-coherent sheaves is defined as the right Kan extension of the functor $A \rightarrow \mathsf{Mod}(A)$ along the inclusion $\mathsf{DAff}\rightarrow \mathsf{PreDStk}$
    \begin{center}
        \begin{tikzcd}[column sep = large]
           \mathsf{DAff} \arrow[d, "\iota"] \arrow[r, "\mathsf{QCoh}|_{\mathsf{DAff}}"] & \mathsf{CAlg}(\mathsf{LPr}) \\
           \mathsf{PreDStk} \arrow[ur, "\mathsf{Ran}_{\iota}\mathsf{QCoh}|_{\mathsf{DAff}}"'],
        \end{tikzcd}
    \end{center}
     i.e. 
        \begin{align}
            \mathsf{QCoh}\coloneqq \mathsf{Ran}_{\iota}\mathsf{QCoh}|_{\mathsf{DAff}} : \mathsf{PreDStk} \longrightarrow \mathsf{CAlg}(\mathsf{LPr}).
        \end{align}
        Let $X$ be a Dirac prestack then the pointwise formula of $\mathsf{QCoh}$ is given by
    \begin{align}
        \mathsf{QCoh}(X) \coloneqq \underset{\mathsf{Spec}(A)\rightarrow X}{\mathsf{lim}}\mathsf{QCoh}(\mathsf{Spec}(A)).
    \end{align}
\end{definition}
The functor $\mathsf{QCoh}$ is naturally equipped with $*$-functors. In general, there is no formula for the $*$-pushforward of a morphism of stacks. Base-change theorems are situations where the $*$-pushforward is computable. 
\begin{lemma}[{\cite[Lemma 3.18.]{hesselholt2024dirac}}]\label[lemma]{basechangeAff}
    Given a cartesian square of affine superschemes
    \[
    \begin{tikzcd}
       T \arrow[d, "f'"] \arrow[r, "g'"] & T  \arrow[d, "f"] \\
       S' \arrow[r, "g"] & S 
    \end{tikzcd},
    \]
    such that $g$ is flat, the diagram 
      \[
    \begin{tikzcd}
       \mathsf{QCoh}(T')   & \arrow[l, "g'^*", swap] \mathsf{QCoh}(T)   \\
       \mathsf{QCoh}(S') \arrow[u, "f'^*", swap] & \arrow[u, "f^*", swap] \arrow[l, "g^*", swap] \mathsf{QCoh}(S) 
    \end{tikzcd},
    \]
is right adjointable in the sense that the canonical map 
\begin{align}
    f^* g_* \longrightarrow g'_* g'^* f^* g_* \simeq g'_* f'^* g^* g_* \longrightarrow g'_* f'^*
\end{align}
is an equivalence.
\end{lemma}

We will also need a more general version of base-change theorem for which we will need the following definitions.
\begin{definition}
    A morphism of Dirac stacks $f: X\rightarrow \mathfrak{Y}$  is \textit{schematic} if, for every morphism $\eta: S \rightarrow Y$, the base change map  $f_S: X\times_Y S$ is equivalent to a morphism of Dirac schemes.
\end{definition}
\begin{definition}
    Let $f:X\rightarrow Y$ be a schematic morphism of Dirac prestacks. Then we call the morphism \textit{quasi-compact} (resp. \textit{quasi-separated}) if the base change along some affine Dirac scheme $X\times_Y S$ is a quasi-compact (resp. quasi-separated) Dirac scheme.
\end{definition}
The following lemma will be important in reducing the next theorem to a local calculation.
\begin{lemma}[{\cite[Proposition 3.6.2.]{gaitsgory2019study}}]\label[lemma]{schematic}
    Let $X$ be a Dirac scheme and $X'\rightarrow X \in \mathsf{DStk}$ be a schematic map. Then $X'$ is also a Dirac scheme. 
\end{lemma}
\begin{proof}
    Let $\underset{i}{\sqcup} S_i \rightarrow X$ be a Zariski atlas. By schematicity each $S_i \times_X X'$ is also a scheme, which we can Zariski cover by $\underset{j \in J_i}{\sqcup}T_j \rightarrow S_i \times_X X'$. Surjectivity is stable by base change, thus we obtain the following commutative diagram 
    \begin{center}
        \begin{tikzcd}
            \underset{i}{\sqcup}\underset{j \in J_i}{\sqcup}T_j \arrow[dr, two heads] \arrow[r, two heads] & \underset{i}{\sqcup} S_i \times_X X' \arrow[r, two heads] \arrow[d, two heads] & \underset{i}{\sqcup} S \arrow[d, two heads]\\
             &X' \arrow[r] & X
        \end{tikzcd}.
    \end{center}
    We claim $ \underset{i}{\sqcup}\underset{j \in J_i}{\sqcup}T_j$ is a Zariski cover of $X'$. This follows by inspecting a single piece of this diagram, i.e.
     \begin{center}
        \begin{tikzcd}
            T_j \arrow[r, hook,"\circ" marking] \arrow[dr, hook, "\circ" marking] &  S_i \times_X X' \arrow[r, hook,"\circ" marking] \arrow[d, hook,"\circ" marking] &  S_i \arrow[d, hook,"\circ" marking]\\
             &X' \arrow[r] & X
        \end{tikzcd},
    \end{center}
    where the hooked arrows with circle mean they are open immersions and the left downwards arrow is an open immersion, because open immersions are stable by base change. Open immersions compose thus $\underset{i}{\sqcup}\underset{j \in J_i}{\sqcup}T_j$ is a Zariski atlas.
\end{proof}

We can now prove a more general version of the base-change theorem, which was omitted in the proof of recollement in \cite{hesselholt2024dirac}, which requires the following lemma.
\begin{lemma} \label[lemma]{pullbackcofinal}
    Let $\mathcal{C}$ be a complete category and consider the commutative diagram
\begin{center}
    \begin{tikzcd}
        I \arrow[rr, "f"] \arrow[dr, "G"'] & &J \arrow[dl, "F"] \\
       & \mathcal{C} &
    \end{tikzcd}
\end{center}
    Furthermore, assume that $f$ fits into an adjunction
    \begin{align}
        f: I \adjoint J: g
    \end{align}
    then $f$ is cofinal, i.e. for $X \in \mathsf{Fun}(J,\mathcal{C})$ there exists an equivalence
    \begin{align}
        \underset{I}{\mathsf{lim}} X \circ f \xlongrightarrow{\sim} \underset{J}{\mathsf{lim}} X.
    \end{align}
\end{lemma}
\begin{theorem}[{\cite[Proposition 2.2.2.]{gaitsgory2019study}}]\label[theorem]{basechangePre}
    Let $f:X \rightarrow Y$ be a flat, schematic, qcqs morphism of Dirac prestacks.  then $f_*$ preserves all colimits and the diagram
\[
\begin{tikzcd}
       \mathsf{QCoh}(X')   & \arrow[l, "f'^*", swap] \mathsf{QCoh}(Y')   \\
       \mathsf{QCoh}(X) \arrow[u, "g'^*", swap] & \arrow[u, "g^*", swap] \arrow[l, "f^*", swap] \mathsf{QCoh}(Y) 
    \end{tikzcd}
\]
is right adjointable in the sense that the base-change map
\begin{align}
    g^* f_* \longrightarrow f'_* f'^* g^* f_* \simeq f'_* g'^* f^* f_* \longrightarrow f'_* g'^*
\end{align}
\end{theorem}
\begin{proof}
    The proof consists of two parts. The first one consists of multiple reduction arguments to reduce to the affine case,. In the affine case we argue by hand.
    For every prestack, it holds that $Y= \underset{S \rightarrow Y}{\mathsf{colim}} S$ and all $S$ can be covered by affines, thus it is enough to check base change along affines, so we can set $Y = \mathsf{Spec}(A)$. Thus, by \cref{schematic} $X$ is also a scheme and by definition of qcqs $X \in \mathsf{Sch}_{\mathsf{qcqs}}$. We observe that we have the factorization  
    \begin{center}
        \begin{tikzcd}[column sep = large, row sep = large]
            \mathsf{DAff}^{\mathsf{op}} \arrow[r, "\mathsf{QCoh}"] \arrow[d, hook, "\mathsf{inc}_1", swap] \arrow[dr, dashed, "\mathsf{inc}"] & \mathsf{LPr} \\
        \mathsf{DSch}_{\mathsf{qcqs}}^{\mathsf{op}} \arrow[r, hook, "\mathsf{inc}_2", swap] &  \mathsf{PreDStk} \arrow[u, "\mathsf{QCoh}", swap]
        \end{tikzcd},
    \end{center} 
    since $\mathsf{DAff} \subseteq \mathsf{DSch}_{\mathsf{qcqs}}$ is a full subcategory. Thus, using transitivity of Kan extensions,
    \begin{align}
        \mathsf{QCoh}(X) \xlongrightarrow{\sim} \underset{S \rightarrow X}{\mathsf{lim}} \mathsf{QCoh}(S),
    \end{align}
     is an equivalence, where $S \rightarrow X \in {(\mathsf{DSch}_{qcqs}}_{ / X})^{\mathsf{op}}$. Furthermore, by \cite[Lemma 6.1.1.1.]{lurie:htt} there exists an adjunction 
    \begin{align}
      \begin{split}
           F: {(\mathsf{DSch}_{qcqs}}_{ / X})^{\mathsf{op}} &\adjoint {(\mathsf{DSch}_{qcqs}}_{ / Y})^{\mathsf{op}} : G \\
       (S\to X)& \longmapsto (S\to X \rightarrow Y) \\
       (S\times_XY)&\longmapsfrom (S\to Y),
      \end{split}
    \end{align}
   which by \cref{pullbackcofinal} induces the equivalence  
    \begin{align}
        \mathsf{QCoh}(X) \xlongrightarrow{\sim} \underset{S \rightarrow X}{\mathsf{lim}} \mathsf{QCoh}(S \times_Y X).
    \end{align}
    By applying \cite[Proposition 4.7.4.19.]{lurie:ha}, it suffices to prove that $f_*$ preserves colimits for a flat morphism between qcqs schemes $X \rightarrow Y$ and show that the natural transformation is an isomorphism when all Dirac prestacks are assumed to be qcqs schemes, i.e. 
    \begin{center}
        \begin{tikzcd}
            X' \arrow[r, "f'"] \arrow[d, "g'"] & Y' \arrow[d, "g"] \\
            X \arrow[r, "f"] & Y 
        \end{tikzcd}.
    \end{center}
     Since $\mathsf{QCoh}$ satisfies fppf descent it follows that it also satisfies Zariski descent, so we only need to verify the base change isomorphism in this case. We will proceed by induction on the number of affines covering $Y'$, which is also finite by the quasi-separatedness assumption. The base case is $Y' = \mathsf{Spec}(A')$, thus $X' = \mathsf{Spec}(B')$ this now follows by \cref{basechangeAff}. Now let $Y' = U_1 \cup U_2$ and denote $U_{1,2} \coloneqq U_1 \times_{Y'} U_2$ and similarly $f_1 \coloneqq f|_{U_1}$, $f_2 \coloneqq f|_{U_2}$ and $f_{1,2}\coloneqq f|_{U_{1,2}}$. By the induction hypothesis, we can assume that the base change isomorphism holds for $f_1$, $f_2$ and $f_{1,2}$. Furthermore, we have 
    \begin{align}
        \begin{split}
            f_*: \mathsf{QCoh}(Y) &\longrightarrow \mathsf{QCoh}(Y') \simeq  \mathsf{QCoh}(U_1) \times_{\mathsf{QCoh}(U_{1,2})} \mathsf{QCoh}(U_2) \\
        \mathcal{F} &\longmapsto f_* \mathcal{F} \simeq  (f_1)_*(\mathcal{F}|_{U_1}) \times_{(f_{1,2})_*(\mathcal{F}|_{U_{1,2}})} (f_2)_*(\mathcal{F}|_{U,2})
        \end{split}
    \end{align}
\end{proof}
\begin{remark}\label[remark]{wecandoderived}
   If we choose to work in the derived setting, i.e. replacing the category of Dirac rings by the category of animated Dirac rings, the same proof holds without the flatness assumption.
\end{remark}
Another property of the category of quasi-coherent sheaves is that it is a sheaf, otherwise also known as descent. As a corollary one also obtain that $\mathsf{QCoh}$ does not differentiate between a prestack and its sheafification. One can also interpret this as $\mathsf{QCoh}$ being insensitive to the topology of a stack.
\begin{theorem}[{\cite[Theorem 3.24.]{hesselholt2024dirac}}] \label[theorem]{QCohIsSheaf}
    Up to contractible choice, there exists a unique factorization 
    \begin{center}
        \begin{tikzcd}
            \mathcal{P}(\mathsf{sAff}^{\mathsf{op}}) \arrow[dr, "\mathsf{QCoh}"'] \arrow[rr, "L"] & & \mathsf{Shv}(\mathsf{sAff})^{\mathsf{op}} \arrow[ld, "\mathsf{QCoh}"]\\
            &\mathsf{CAlg}(\mathsf{LPr}) &
        \end{tikzcd}
    \end{center}
    of the functor $\mathsf{QCoh}: \mathcal{P}(\mathsf{sAff}^{\mathsf{op}}) \rightarrow \mathsf{CAlg}(\mathsf{LPr})$ and the functor 
    \begin{align}
        \mathsf{QCoh}: \mathsf{Shv}(\mathsf{sAff})^{\mathsf{op}} \rightarrow \mathsf{CAlg}(\mathsf{LPr})
    \end{align}
    takes small colimits of Dirac stacks to limits of symmetric monoidal categories. In particular, let $X \in \mathcal{P}(\mathsf{sAff}$ we obtain that the unit of the sheafification adjunction $\eta: \iota \circ L \rightarrow \mathsf{id}$ induces an equivalence of categories 
    \begin{align}
        \mathsf{QCoh}((\iota \circ L)(X)) \xlongrightarrow{\eta^*} \mathsf{QCoh}(X).
    \end{align}
\end{theorem}
The following theorem is the algebro-geometric version of the open-closed decomposition and a key result needed to prove that crystals are equivalent to $D$-modxules.
\begin{theorem}[cf. {\cite[Theorem 4.3.]{hesselholt2024dirac}}] \label[theorem]{recollement}
    Let $X$ be a Dirac scheme, and let $g: Z \rightarrow X$ be a closed immersion. Let $j: X_Z^\wedge \rightarrow X$ be the formal completion of $X$ along $g$, and let $i: U \simeq  X - Z \rightarrow X$ be the inclusion of the open complement of $Z$. In this situation, we have
    \begin{center}
        \begin{tikzcd}
\mathsf{QCoh}(U) \arrow[r, "i_* \simeq i_!"] 
& \mathsf{QCoh}(X) \arrow[l, bend right=40, "i^*"'] \arrow[l, bend left=40, "i^!"'] \arrow[r, "j^! \simeq j^*"] 
& \mathsf{QCoh}(X_Z^\wedge) \arrow[l, bend right=40, "j_!"'] \arrow[l, bend left=40, "j_*"']
\end{tikzcd}
    \end{center}
    and $j_*$ is $t$-exact.
\end{theorem}

\subsection{Reduction from Dirac Geometry}
We consider the category of affine superschemes $\mathsf{sAff}$ and the category of dirac schemes $\mathsf{DAff}$. The categories of Dirac prestacks and pre-superstacks are closely related, which arises from the adjunction
\begin{align}
        U^* : \mathcal{P}(\mathsf{sAff}) \adjoint \mathcal{P}(\mathsf{DAff}): P^*
    \end{align}
    induced by 
    \begin{align}
  P:  \mathsf{DCAlg} \adjoint \mathsf{sCAlg}: U.
\end{align}
Let us first prove this is an adjunction. In particular, we will prove that there is an adjunction between the categories of $\mathbb{Z}$-graded abelian groups and $\mathbb{Z}_2$-graded abelian groups. The adjunction is monoidal and hence restricts to the categories of commutative algebras. To do this, let us quickly review the concept of Day convolution. 
\begin{definition}
    Let $(\mathcal{C}, \otimes_\mathcal{C}, \mathbbm{1}_\mathcal{C}),(\mathcal{D}, \otimes_\mathcal{D}, \mathbbm{1}_\mathcal{D})$ be symmetric monoidal categories. The \textit{external tensor product} is given by yhe composition:
    \begin{center}
        \begin{tikzcd}
            \mathsf{Fun}(\mathcal{C},\mathcal{D}) \times \mathsf{Fun}(\mathcal{C},\mathcal{D}) \arrow[r, hook] \arrow[rr, bend right = 20, " (- \boxtimes-) "'] & \mathsf{Fun}(\mathcal{C}\times \mathcal{C},\mathcal{D}\times \mathcal{D}) \arrow[r, "\otimes_\mathcal{D}"] & \mathsf{Fun}(\mathcal{C} \times \mathcal{C},\mathcal{D}).
        \end{tikzcd}
    \end{center}

\end{definition}
\begin{definition}[cf. {\cite[Definition 2.8.]{glasman2017day}}]
    Let $(\mathcal{C}, \otimes_\mathcal{C}, \mathbbm{1}_\mathcal{C}),(\mathcal{D}, \otimes_\mathcal{D}, \mathbbm{1}_\mathcal{D})$ be symmetric monoidal categories, where $-\otimes_\mathcal{D}-$ preserves colimits in each variable and let $\mathcal{D}$ be cocomplete. The \textit{Day Convolution} is a symmetric monoidal structure on $\mathsf{Fun}(\mathcal{C},\mathcal{D})$, where the multiplication is given by the following composition:
    \begin{center}
        \begin{tikzcd}[column sep = huge]
            \mathsf{Fun}(\mathcal{C}\times \mathcal{C},\mathcal{D}\times \mathcal{D}) \arrow[rr, bend right = 22, "\otimes_\mathsf{Day}"'] \arrow[r, "\otimes_\mathcal{D}"] &\mathsf{Fun}(\mathcal{C}\times \mathcal{C}, \mathcal{D}) \arrow[r, shift right = 2, "(\otimes_\mathcal{C})_!"', "\perp"] & \mathsf{Fun}(\mathcal{C},\mathcal{D}) \arrow[l, shift right =2, "(\otimes_\mathcal{C})^*"'].
        \end{tikzcd}
    \end{center}
    As it is defined by the left Kan extension, there exists a pointwise formula. Let $X \in \mathcal{C}$, then we have 
    \begin{align}
        (F\otimes G)(X) \coloneqq \underset{X_1 \otimes_\mathcal{C}X_2 \to X}{\mathsf{colim}}F(X_1)\otimes_\mathcal{D}G(X_2).
    \end{align}
\end{definition}
Furthermore, given an abelian group $G$, we can also identify the category of $G$-graded abelian groups as 
\begin{align}
   \mathsf{Ab}_G\simeq  \mathsf{Fun}(G,\mathsf{Ab}).
\end{align}
Moreover, $G$ can be viewed as a $0$-truncated anima equipped with a symmetric monoidal structure and similarly, if $f:G\rightarrow H$ is a morphism of abelian groups then it can also be viewed as a symmetric monoidal functor. Thus, we can equip $\mathsf{Ab}_G$ with a symmetric monoidal structure given by Day convolution. We will now prove that, under mild hypothesis, $f:G\to H$ induces a symmetric monoidal adjunction 
\begin{align}
    f_!: \mathsf{Ab}_G \adjoint \mathsf{Ab}_H: f^*.
\end{align}
\begin{proposition} \label[proposition]{FunwithDay}
    Let $f: G \to H$ be a morphism of abelian groups and let $\mathcal{C}$ be a symmetric monoidal category, whose tensor product $\otimes_\mathcal{C}$ preserves colimits in each variable. Then the functor 
    \begin{align}
        f_!: \mathsf{Fun}(G,\mathcal{C}) \longrightarrow \mathsf{Fun}(H,\mathcal{C})
    \end{align}
    is symmetric monoidal with respect to the Day convolution. In particular, the adjunction $f_! \dashv f^*$ is a monoidal adjunction.
\end{proposition}
\begin{proof}
    Let $\mathcal{F}, \mathcal{G} \in \mathsf{Fun}(G,\mathcal{C})$, by denoting the group multiplication of $G$ and $H$ by $\otimes_G$ and $\otimes_H$ respectively, we obtain the following formulas 
    \begin{align}
            f_!(\mathcal{F}\otimes_\mathsf{Day}\mathcal{G}) &\simeq \mathsf{Lan}_{f} \circ \mathsf{Lan}_{\otimes_G}\circ (-\otimes_\mathcal{C}-)\circ (\mathcal{F}\times G) \\
           f_!\mathcal{F}\otimes_\mathsf{Day}f_!\mathcal{G}  &\simeq  \mathsf{Lan}_{\otimes_H}\circ (-\otimes_\mathcal{C}-)\circ \mathsf{Lan}_{f\times f} (\mathcal{F} \times \mathcal{G}). 
    \end{align}
    One can now compute the following:
    \begin{align}
        \begin{split}
    f_!\mathcal{F}\otimes_\mathsf{Day}f_!\mathcal{G}  &\simeq  \mathsf{Lan}_{\otimes_H}((-\otimes_\mathcal{C}-)\circ \mathsf{Lan}_{f\times f} (\mathcal{F} \times \mathcal{G}) \\
    & \simeq \mathsf{Lan}_{\otimes_H}\circ \mathsf{Lan}_{f\times f}\circ (-\otimes_\mathcal{C}-)\circ (\mathcal{F} \times \mathcal{G}) \ \ (\otimes_\mathcal{C} \ \text{is preserves colimits}) \\
    & \simeq \mathsf{Lan}_{\otimes_H \circ f \times f} \circ (-\otimes_\mathcal{C}-)\circ (\mathcal{F} \times \mathcal{G}) \ \ (\text{Kan extensions compose})
        \end{split}
    \end{align}
    It remains to show that $\mathsf{Lan}_{\otimes_H \circ f \times f} \simeq \mathsf{Lan}_{f\circ \otimes_G}$. Indeed, this follows from the fact that since $f:G \rightarrow H$ is a group homomorphism, the following diagram commutes: 
    \begin{center}
        \begin{tikzcd}
        G \times G \arrow[r, "f \times f"] \arrow[d, "\otimes_G"] & H \times H \arrow[d, "\otimes_H"] \\
            G \arrow[r, "f"] & H, 
        \end{tikzcd}
    \end{center}
    i.e. $(-\otimes_H -)\circ f \times f \simeq f \circ (-\otimes_G-)$ and the claim follows.
\end{proof}
\begin{corollary}
    There exists an adjunction
    \begin{align}
        P: \mathsf{DCAlg} \adjoint \mathsf{sCAlg}: U,
    \end{align}
    given by folding the $\mathbb{Z}$-grading and by extending the $\mathbb{Z}_2$-grading to a 2-periodic $\mathbb{Z}$-grading.
\end{corollary}
\begin{proof}
    We apply \cref{FunwithDay} for $f: \mathbb{Z} \to \mathbb{Z}_2$ the unique non-trivial group homomorphism. Due to monoidality of the adjunction, it restricts to commutative algebra objects, which completes the proof.
\end{proof}
One now also identifies $\mathsf{sAff}$ to be the non-full subcategory of $2$-periodic affine Dirac schemes, so we may view superschemes as special types of Dirac schemes. Let us now prove that the additional maps that exist in the category of Dirac prestacks do not affect the geometry of sheaves of a pre-superstack, when viewed as a $2$-periodic Dirac prestack. Let $\mathsf{sAff}$ be the non-full subcategory of 2-periodic Dirac affines. Furthermore, denote $2\mathsf{DAff}$ the full subcategory of $2$-periodic Dirac affine, which is the essential image of $U:\mathsf{sAff}\to \mathsf{DAff}$. 
\begin{lemma}\label[lemma]{supercofinal}
    Let $X \in \mathcal{P}(\mathsf{sAff})$, then 
    \begin{align}
        U_!: \mathcal{P}(\mathsf{sAff}) \longrightarrow \mathcal{P}(2\mathsf{DAff})
    \end{align}
    induces a map 
    \begin{align}
     U_! (\underset{\mathsf{sAff}_{/X}}{\mathsf{colim}} \mathsf{Spec}(A))  \longrightarrow \underset{2\mathsf{DAff}_{/U_!X}}{\mathsf{colim}} \mathsf{Spec}(B),
    \end{align}
    which is an equivalence.
\end{lemma}
\begin{proof}
   Let $Y\coloneqq U_!(X)$. Expanding the tautological colimit in both categories, we obtain 
    \begin{align}
        U_!(\mathsf{colim}_{\mathsf{sAff}_{/X}}\mathsf{Spec}(A)) \simeq \mathsf{colim}_{\mathsf{2DAff}_{/Y}} \mathsf{Spec}(B). 
    \end{align}
    Since $U_!$ is a left adjoint it commutes with colimits, thus 
    \begin{align}
        \mathsf{colim}_{\mathsf{sAff}_{/U_!X}}U_!\mathsf{Spec}(A) \simeq  \mathsf{colim}_{\mathsf{sAff}_{/U_!X}}\mathsf{Spec}(A) \simeq \mathsf{colim}_{\mathsf{2DAff}_{/U_!X}} \mathsf{Spec}(B),
    \end{align}
    as $U_!$ is essentially surjective on the Yoneda image.
\end{proof}
Let us now identify the quasi-coherent supersheaves with the two periodic quasi-coherent Dirac sheaves.
\begin{lemma}\label[lemma]{superandgraded}
    Let $A$ be a superring then 
    \begin{align}
        \mathsf{sMod}(A) \longrightarrow \mathsf{grMod}(A[u^{\pm}])
    \end{align}
    is a monoidal equivalence. 
\end{lemma}
\begin{proof}
    Both categories are produced by animation, thus it suffices to prove this equivalence $1$-categorically, which we do by checking monoidality, fully faithfulness and essential surjectivity. By unraveling the definitions, we can rewrite the functor as
    \begin{align}
        f^*: \mathsf{Mod}_A(\mathsf{Ab}_{\mathbb{Z}_2}) \longrightarrow \mathsf{Mod}_{f^*A}(\mathsf{Ab}_\mathbb{Z}),
    \end{align}
    which makes it obvious, that it lax monoidal. Checking monoidality is a condition, thus let $M, N \in \mathsf{Mod}_A(\mathsf{Ab}_{\mathbb{Z}_2})$ and we compute
    \begin{align}
        f^*M \otimes_{f^*A} f^*N \simeq M[u^{\pm 1}] \otimes_{A[u^{\pm 1}]} N[u^{\pm 1}] \simeq (M \otimes_A N)[u^{\pm 1}] \simeq f^*(M \otimes_A N).
    \end{align}
    Now, we need to check this functor is fully faithful and essentially surjective. Consider a graded module $\{M_n\}_{n \in \mathbb
    Z}$ over $A[u^{\pm}]$. We need to prove that there exists an isomorphism 
    \begin{align}
        M_n \simeq M_{n+2} \ \forall n.
    \end{align}
    This follows from observing that the action of $u$ on $M$ induces a morphism 
    \begin{align}
        \mu_u: M_n \longrightarrow M_{n+2}
    \end{align}
    which is an isomorphism, since $u$ is invertible. This proves essential surjectivity and faithfulness. To prove fullness consider a morphism between $2$-periodic graded modules $\{M_n\}_{n \in \mathbb
    Z}\rightarrow \{N_n\}_{n \in \mathbb
    Z}$ over $A[u^{\pm}]$. We need to prove that there is an isomorphism $f_n \simeq f_{n+2} \ \forall n$. By definition, we have a commutative diagram
    \begin{center}
        \begin{tikzcd}
            M_n \arrow[r, "f_n"] \arrow[d, "\mu_u"] & N_n \arrow[d,"\mu_u"] \\
            M_{n+2} \arrow[r, "f_{n+2}"] & N_{n+2} 
        \end{tikzcd}
    \end{center}
    from which we obtain the equation 
    \begin{align}
        f_{n+2} \circ \mu_u \simeq \mu_u \circ f_n
    \end{align}
    and one can compute for an element $m \in M_{n+2}$
    \begin{align}
        f_{n+2}(m) = \mu_u(f_n(\mu_{u^{-1}}(m))) = uf_n(u^{-1}m) = f_n(m).
    \end{align}
    This proves fullness and completes the proof.
\end{proof}
\begin{theorem}
    Let $X$ be a superstack. Then there exists an equivalence of symmetric monoidal categories 
    \begin{align}
        \mathsf{sQCoh}(X) \simeq \mathsf{DQCoh}(U_!X)
    \end{align}
\end{theorem}
\begin{proof}
    We may view $X$ as a Dirac prestack $U_!X$ and is a 2-periodic presheaf, hence, by an analogous argument as in \cref{supercofinal}, we may also index the tautological colimit over $\mathsf{2DAff}_{/U_!X}$. By definition, it is an accessible presheaf, which means it can be written as a small colimit indexed by $\mathsf{2DAff}_{/U_!X}$. By \cref{supercofinal} we may reindex this diagram by $\mathsf{sAff}_{/X}$ and we obtain
    \begin{align}
    \begin{split}
        \mathsf{QCoh}(U_!X) & \coloneqq \mathsf{QCoh}(\underset{\mathsf{DAff}_{/X}}{\mathsf{colim}}\mathsf{Spec}(A)) \simeq  \underset{\mathsf{sAff}_X^{\mathsf{op}}}{\mathsf{lim}}\mathsf{QCoh}(\mathsf{Spec}(A)) \simeq \underset{\mathsf{sAff}_X^{\mathsf{op}}}{\mathsf{lim}} \mathsf{grMod}(A) \\
        & \simeq \underset{\mathsf{sAff}_X^{\mathsf{op}}}{\mathsf{lim}} \mathsf{sMod}(A),
    \end{split}
    \end{align}
    where the first equivalence is by definition, the second equivalence follows from \cref{QCohIsSheaf} and the last equivalence follows from \cref{superandgraded}.
\end{proof}
As a consequence all properties that were proven for quasi-coherent sheaves in the Dirac setting of \cite{hesselholt2024dirac} hold for superstacks.
\subsection{Stacks as Colocalization}
In supergeometry, there is a triple of adjoint functors 
\begin{center}
    \begin{tikzcd}[column sep=large]
        \mathsf{Aff} \arrow[rr, "\iota"] & & \mathsf{sAff} \arrow[ll, bend left, "G"] \arrow[ll, bend right, "F"']
    \end{tikzcd}
\end{center}
where $\iota$ is the inclusion, $F$ is taking the even part and $G$ is quotienting out by the odd part. In this section, we show how to construct these functors from general methods involving Kan extensions and extend this adjunction to superstacks, which is a proof of \cite[Remark 2.28.]{bruzzo2025foundations}. This section is inspired by \cite{390095} and we note that for the purposes of this section we will be working in the étale topology. \\
    We will work over an algebraically closed field of characteristic $0$, which we denote by $k$. Let $\mathsf{Vect}$ and $\mathsf{sVect}$ denote the category of k-vector spaces and super vector spaces respectively. We observe that we have an obvious adjunction 
\begin{center}
        \begin{tikzcd}
        \mathsf{Vect} \arrow[r, shift right =2 , "\iota_{\mathsf{Aff}}^{\mathsf{op}}"'] &  \arrow[l, shift right = 2, "F^{\mathsf{op}}"'] \mathsf{sVect}.
    \end{tikzcd}
\end{center}
The inclusion is the left adjoint and it is symmetric monoidal, thus this adjunction descends to commutative algebra objects

\begin{center}
        \begin{tikzcd}
        \mathsf{CAlg}(\mathsf{Vect}) \arrow[r, shift right =2 , "\iota_{\mathsf{Aff}}^{\mathsf{op}}"'] &  \arrow[l, shift right = 2, "F^{\mathsf{op}}"'] \mathsf{CAlg}(\mathsf{sVect}).
    \end{tikzcd}
\end{center}
We observe that as the inclusion preserves both limits and colimits and thus also admits a left adjoint. We can identify both sides as $\mathsf{CAlg}(\mathsf{Vect}) \simeq \mathsf{CAlg}$  and $\mathsf{CAlg}(\mathsf{sVect}) \simeq \mathsf{sCAlg}$. Moreover, by taking opposite categories we obtain the famous adjunction
\begin{center}
    \begin{tikzcd}
        \mathsf{Aff} \simeq \mathsf{CAlg}^{\mathsf{op}} \arrow[r, "\iota_{\mathsf{Aff}}"] & \mathsf{sCAlg}^{\mathsf{op}} \simeq \mathsf{sAff} \arrow[l, bend left, "G"] \arrow[l, bend right, "F"'].
    \end{tikzcd}
\end{center}
We can extend this to schemes and stacks by taking presheaf categories on both sides our adjunction gets propagated and we obtain even more adjunctions.

\begin{center}
    \begin{tikzcd}[column sep=40pt]
        \mathsf{Fun}(\mathsf{Aff}^{\mathsf{op}}, \mathsf{Ani}) 
        \arrow[rr, shift right=5, "G^*" description] 
        \arrow[rr, shift left=5, "F^*" description] 
        &[-20pt] & 
        \mathsf{Fun}(\mathsf{sAff}^{\mathsf{op}}, \mathsf{Ani})
        \arrow[ll, shift right=10, "\mathsf{Lan}_F" description] 
        \arrow[ll, "\iota_\mathsf{sAff}^*" description]
        \arrow[ll, shift left = 10, "\mathsf{Ran}_G" description]
    \end{tikzcd}
\end{center}
Since the left Kan extension is defined to be the left adjoint of the $*$-pullback functors, we can make the identifications $L^* = \mathsf{Lan}_{\iota}$ and $\iota^* =\mathsf{Lan}_R$. We have also exhibited the category of stacks to be a localization, i.e. it is a full subcategory, and it is equipped with a sheafification functor. Therefore, we can extend this diagram
\begin{center}
    \begin{tikzcd}[column sep=40pt]
    \mathsf{Stk} \arrow[r, "\mathsf{inc}_\mathsf{Stk}" description, shift right= 2]  & \arrow[l, shift right= 2, "L_{\mathsf{Stk}}" description] \mathsf{Fun}(\mathsf{Aff}^{\mathsf{op}}, \mathsf{Ani}) 
        \arrow[rr, shift right=5, "G^*" description] 
        \arrow[rr, shift left=5, "F^*" description] 
        &[-20pt] & 
        \mathsf{Fun}(\mathsf{sAff}^{\mathsf{op}}, \mathsf{Ani})
        \arrow[ll, shift right=10, "\mathsf{Lan}_F" description] 
        \arrow[ll, "\iota_\mathsf{Aff}^*" description]
        \arrow[ll, shift left = 10, "\mathsf{Ran}_G" description] \arrow[r, shift right= 2, "L_\mathsf{sStk}" description] & \arrow[l, "\mathsf{inc}_\mathsf{Stk}" description,shift right = 2] \mathsf{sStk}
    \end{tikzcd},
\end{center}
where we denote both sheafification functors by $L$ and their respective subscript. The sheafification functors are also left adjoints, and as the top three middle arrows are all left Kan extensions, hence all left adjoints, we can compose them with the sheafification functors and we obtain an adjoint triple in the superstack case. 

\begin{center}
    \begin{tikzcd}[column sep=large]
        \mathsf{Stk} \arrow[rr, "\iota"] & & \mathsf{sStk} \arrow[ll, bend left, "(-)_{\mathsf{bos}}"] \arrow[ll, bend right, "L"']
    \end{tikzcd}
\end{center}
One now needs to analyze whether this adjunction agrees with the adjunction in the scheme case, which is what we are going to do now.
\begin{proposition}
    The functor $\iota$ admits a right adjoint given by $(-)_{\mathsf{bos}}$. 
\end{proposition}
\begin{proof}
    Having constructed the functors it remains to check, how they behave on objects. First, we observe $\iota$ is the composite funtor $ L_\mathsf{sStk}\circ F^* \circ \mathsf{inc}_\mathsf{Stk}: \mathsf{Stk} \rightarrow \mathsf{sStk}$. One can now compute 
    \begin{align}
    \begin{split}
        \mathsf{Map}_{\mathsf{sStk}}(L_{\mathsf{sStk}}(F^*((\mathsf{inc}_\mathsf{Stk}(X))),Y) &\simeq  \mathsf{Map}_{\mathsf{PreStk}}(F^*(\mathsf{inc}_\mathsf{Stk}(X)),\mathsf{inc}_\mathsf{sStk}(Y)) \\
        & \simeq  \mathsf{Map}_{\mathsf{PreStk}}(\mathsf{inc}_\mathsf{Stk}(X), \iota_\mathsf{Aff}^*(\mathsf{inc}_\mathsf{sStk}(Y))) \\
        & \simeq  \mathsf{Map}_{\mathsf{Stk}}(X, \iota_\mathsf{Aff}^*(Y))
    \end{split},
    \end{align}
    where the first equivalence follows from using the right adjoint of the sheafification $L_\mathsf{sStk}$, the second equivalence follows from the adjointness of $F^* \dashv \iota_\mathsf{Aff}^*$ and the last equivalence follows from the fact that the inclusion functors are both fully faithful, and the restriction functor $\iota_\mathsf{Aff}^*$ is precomposition with the functor $\iota_\mathsf{Aff}$, which means that given a superstack $Y$, we will then simply only evaluate $Y$ on ordinary schemes, hence it remains a stack. Since $\mathsf{Stk}$ forms a full subcategory of $\mathsf{PreStk}$, we can swap out the subscript $\mathsf{PreStk}$ by $\mathsf{Stk}$, since both objects in the mapping space are stacks. Therefore, $\iota$ admits a right adjoint given by what we denote by $(-)_{\mathsf{bos}}$.
\end{proof}
To prove the fully faithfulness, let us admit the following characterization of fully faithful functors
\begin{proposition}[cf. {\cite{carmeli2021ambidexterity}}]
    fully faithful if and only if $\mathsf{id}\to RL$ is an equivalence
\end{proposition}
\begin{proposition}
    $\iota: \mathsf{Stk} \to \mathsf{sStk}$ is fully faithful.
\end{proposition}
\begin{proof}
    Let us first compute how the composite functor $(-)_\mathsf{bos}\circ\iota$ acts on a stack $X$. One computes that 
    \begin{align}
        L_{\mathsf{sStk}}F^*\mathsf{inc}_{\mathsf{Stk}}(X) \simeq L_{\mathsf{sStk}} \underset{\mathsf{Spec}A\to X}{\mathsf{colim}}F^*\mathsf{Spec}A \simeq X,
    \end{align}
    where the first equivalence follows from the fact that $F^*$ is a Left Kan extension and hence preserves colimits, which is furthermore fully faithful. The second equivalence follows from the fact that sheafification is a left adjoint.
    We continue to compute composition with $(-)_\mathsf{bos}$:
    \begin{align}
       L_\mathsf{Stk} \iota_\mathsf{Aff}^*\mathsf{inc}_\mathsf{sStk}X \simeq  L_\mathsf{Stk}\underset{\mathsf{Spec}A\to X}{\mathsf{colim}}\iota_\mathsf{Aff}^*\mathsf{Spec}A \simeq X.
    \end{align}
    We will now argue, why the unit map is an equivalence. The unit map is defined pointwise for two objects $X,Y$ as the map that induces the following equivalence.
    \begin{center}
        \begin{tikzcd}
            \mathsf{Map}_\mathsf{sStk}(\iota X,Y) \arrow[r] \arrow[dr, "\sim"] & \mathsf{Map}((\iota X)_\mathsf{bos}, Y_\mathsf{bos}) \arrow[d, "\eta^*_X"] \\
            & \mathsf{Map}_{\mathsf{Stk}}(X, Y_\mathsf{bos}) .
        \end{tikzcd}
    \end{center}
    We can identify $\mathsf{Map}((\iota X)_\mathsf{bos}, Y_\mathsf{bos}) \simeq \mathsf{Map}(X,Y_\mathsf{bos})$. To show $\eta_X$ must be an equivalence, we will argue by contradiction. Suppose $\eta_X$ were not an equivalence, then the composition cannot be an equivalence as well, which is a contradiction as $\iota $ is left adjoint to $(-)_\mathsf{bos}$, thus $\eta_X$ is an equivalence. Furthermore, it suffices to show that a candidate unit map induces said equivalence, where one observes that the identity does the job. Therefore, 
    \begin{align}
        \mathsf{id} \xlongrightarrow{} (-)_\mathsf{bos} \iota
    \end{align}
    is an equivalence, as wished.
\end{proof}
\section{Transmutation}
In this section we construct the transmutation corresponding to de Rham, Betti and Dolbeault cohomology in the setting of supergeometry. 
\begin{definition}
    Let $R$ be a ring object in the category of derived stacks $\mathsf{dStk}$. Let $X$ be a derived scheme, then we define 
    \begin{align}
        X_R(S) \coloneqq X(R(S))
    \end{align}
    as the \textit{R-transmutation} of $X$.
\end{definition}
These stacks were first defined by Simpson in \cite{simpson1999algebraic} in his pursuit of higher non-abelian Hodge theory. Let $X$ be a scheme then we denote by $X_{A}$, with $A \in \{\mathsf{dR},\mathsf{B},\mathsf{Dol}\}$, the stack corresponding to the cohomology theory. Using this perspective one also obtains a uniform description of the three moduli stacks in nonabelian Hodge theory as mapping stacks
\begin{align}
    \mathsf{Map}(X_A,BG).
\end{align}
Another feature of this approach is that it internalizes extra data of sheaves. More concretely, we will shortly be discussing $D$-modules, locally constant sheaves and Higgs sheaves. All of these objects are special types of sheaves and using transmutations, we can realize all of these objects as quasi-coherent sheaves on their respective transmutation stack. In this section, we will see that quasi-coherent sheaves on de Rham stack will correspond to $D$-modules, quasi-coherent sheaves on the Betti stack will correspond to locally constant sheaves and quasi-coherent sheaves on the Dolbeault stack will correspond to Higgs sheaves. In this entire section, we will work in the étale topology and over an algebraically closed field of characteristic zero $k$.

\subsection{D-Modules}
We use our geometric intuition of interpreting supergeometry as a variant of derived algebraic geometry and adopt their definition of $D$-modules, which can be defined on arbitrary derived prestacks. One then verifies definition motivated by derived correctly produces $D$-supermodules, in the case of smooth superschemes and smooth Artin superstacks. We also note that $D$-modules were already constructed by Carchedi in the setting of derived differental supergeometry \cite{carchedi2025quasicoherentsheavesdmodulesderived} defining the de Rham stack using a topos-theoretic approach similar to \cite{PortaSalaShapes}. 
The theory of $D$-modules runs into technical difficulties, when one tries to generalize to stacks and non-smooth schemes. Consequently, in the setting of derived algebraic geometry, one replaces them by the category of crystals. We will follow this description of $D$-modules using the category of crystals as a replacement and show that, in the case of smooth Artin superstacks, it is the category of $D$-modules. Furthermore, this viewpoint also allows us to reprove and generalize, in a simple fashion,  a classical theorem due to Penkov, which states:
\begin{theorem}[\cite{penkov1983d}]
    Let $X$ be a complex supermanifold. Then the category of D-supermodules on $X$ is equivalent to the category of $D$-supermodules on the bosonic truncation.
\end{theorem}
At the end of the section, we give some examples of categories of $D$-modules and show that coherent cohomology of the de Rham stack computes de Rham cohomology, where we obtain the isomorphism between de Rham cohomology of the underlying bosonic space and the de Rham cohomology of the superspace as a corollary. We will follow \cite{gaitsgory2023crystals}.
To see precisely, how we adapt the definition of de Rham stack we will give the definition for derived prestacks first. 
\begin{definition}[{\cite[Definition 1.1.1.]{gaitsgory2023crystals}}]
    Let $X$ be a derived prestack. Then we define the de Rham prestack of $X$ to be 
    \begin{align}
        X_{\mathsf{dR}}(S) \coloneqq X((\pi_0S)_{\mathsf{red}}).
    \end{align}
\end{definition}
Next, we adapt this definition to the super setting. Recalling, that both the category of derived schemes and the category of superschemes contain the category of schemes as a colocalization, we can give the following definition.
\begin{definition}
    Let $X$ be a pre-superstack. Then we define the de Rham prestack of $X$ to be 
    \begin{align}
        X_{\mathsf{dR}}(S) \coloneqq X(S_{\mathsf{br}}).
    \end{align}
    where $(-)_{\mathsf{br}}$ denotes the functor $(-)_{\mathsf{red}}\circ(-)_{\mathsf{bos}}:\mathsf{sSch} \rightarrow \mathsf{Sch}_{\mathsf{red}}$.
\end{definition}
 Alternatively, for a prestack $X$ we define $\bar{X}\coloneqq X|_{\mathsf{Aff}_{\mathsf{red}}}: \mathsf{Aff}_{\mathsf{red}} \rightarrow \mathsf{Ani}$ to be the restriction of the functor of points to the reduced bosonic schemes. Then 
\begin{center}
    \begin{tikzcd}
        \mathsf{Aff}_{\mathsf{red}} \arrow[d, "\iota"] \arrow[r] & \mathsf{Ani} \\
        \mathsf{Aff} \arrow[ur, swap, "\mathsf{Ran}_\iota \bar{X}"]
    \end{tikzcd}
\end{center}
and we define $X_{\mathsf{dR}} \coloneqq \mathsf{Ran}_\iota \bar{X}$.
\begin{lemma}[{\cite[Lemma 1.1.4.]{gaitsgory2023crystals}}]\label[lemma]{limitscolimits}
    The functor $\mathsf{dR}: \mathsf{PresStk} \rightarrow \mathsf{PresStk}$ commutes with limits and colimits
\end{lemma}
\begin{proof}
    Since $\mathsf{PresStk} = \mathsf{Fun}(\mathsf{sAff}^{\mathsf{op}}, \mathsf{Ani})$ and limits and colimits in functor categories are computed objectwise, the claim follows from the fact that $\mathsf{Ani}$ contains all limits and colimits.
\end{proof}
\begin{corollary}[{\cite[Corollary 1.1.5.]{gaitsgory2023crystals}}]\label[corollary]{KanPrestk}
    The functor $\mathsf{dR}: \mathsf{PresStk} \rightarrow \mathsf{PresStk}$ is the left Kan extension of the functor $\mathsf{dR}|_{\mathsf{sAff}}: \mathsf{sAff} \rightarrow \mathsf{PresStk}$ along the inclusion $\mathsf{sAff} \rightarrow \mathsf{PresStk}$.
\end{corollary}
\begin{proof}
The pointwise formula for the left Kan extension is 
\begin{align}
    \mathsf{Lan}_{\mathsf{inc}}\mathsf{dR} (X) \coloneqq \underset{\mathsf{Spec A} \rightarrow X}{\mathsf{colim}} (\mathsf{Spec}(A))_{\mathsf{dR}}
\end{align}
Furthermore, every pre-superstack can be written as
\begin{align}
    \underset{S \rightarrow X}{\mathsf{colim}} S.
\end{align}
The functor $\mathsf{dR}$ is colimit preserving and one observes that both formulae are equivalent.
\end{proof}
\begin{lemma}[{\cite[Lemma 1.1.7.]{gaitsgory2023crystals}}]\label[lemma]{IncKan}
    The functor $\mathsf{dR}|_{\mathsf{sAff}}$ is isomorphic to the left Kan extension of the functor $\mathsf{dR}|_{\mathsf{sAff}_{\mathsf{red}}}$ along the natural inclusion functor.
\end{lemma}
\begin{proof}
The inclusions $\mathsf{Aff}_{\mathsf{red}} \rightarrow \mathsf{Aff}$ and $\mathsf{Aff} \rightarrow \mathsf{sAff}$ are both colimit preserving and thus admit a right adjoint. The composition of adjoint functors is an adjoint functor. Thus,
\begin{align}
    \mathsf{inc}: \mathsf{Aff}_{\mathsf{red}} \longrightarrow \mathsf{sAff}
\end{align} 
admits a right adjoint. We observe the following 
\begin{align}
   \mathsf{inc}^* \vdash R^*: \mathsf{Fun}(\mathsf{PresStk},\mathsf{PresStk}) \adjoint \mathsf{Fun}(\mathsf{sAff},\mathsf{PresStk}),
\end{align}
but the left adjoint of $\mathsf{inc}^*$ is the left Kan extension and adjoints are unique. So the left Kan extension identifies with precomposition by $R^*$ to compute the left Kan extension.
Next, we have to show that the natural transformation
\begin{align}
    \mathsf{Lan}_{\mathsf{inc}}\mathsf{dR}|_{\mathsf{Aff}_{\mathsf{red}}} \rightarrow \mathsf{dR}|_{\mathsf{sAff}}
\end{align}
is an isomorphism. Due to the previous argument, we only have to show that
\begin{align}
    \mathsf{dR}(S_{\mathsf{br}}) \rightarrow \mathsf{dR}(S) \ \forall S \in \mathsf{sAff}
\end{align}
is an isomorphism, which is true by the definition of $\mathsf{dR}$.
\end{proof}
\begin{remark}
    This now also holds for any subcategory pair in the following list: 
    \begin{align}
        \mathsf{Aff}_{\mathsf{red}},\ \mathsf{Aff}, \ \mathsf{sAff}, \ \mathsf{sSch}_{\mathsf{qsqc}} , \ \mathsf{sSch} ,\ \mathsf{PresStk}.
    \end{align}
\end{remark}
We always have the inclusion map $S_{\mathsf{br}} \rightarrow S$, from which we can construct a natural transformation 
\begin{align}
    X(S) \rightarrow X(S_{\mathsf{br}}) = X_{\mathsf{dR}}(S).
\end{align}
So, the functor comes equipped with a natural transformation. 
\begin{align}
    p: \mathsf{id} \rightarrow \mathsf{dR}
\end{align}
\begin{definition}
    Let $X$ be a pre-superstack. We call the category of quasi-coherent sheaves on the associated de Rham prestack the category of \textit{left crystals} and denote it 
    \begin{align}
        \mathsf{Crys}(X) \coloneqq \mathsf{QCoh}(X_{\mathsf{dR}}).
    \end{align}
\end{definition}
We can also write this more functorially as 
\begin{align}
    \mathsf{Crys}\coloneqq \mathsf{QCoh} \circ \mathsf{dR}: \mathsf{PresStk} \rightarrow \mathsf{Cat}_{\infty}.
\end{align}
\begin{remark}
    To do right crystals one needs the theory of $\mathsf{IndCoh}$, whose setup is outside the scope of this paper.
\end{remark}
 We are now already in a position to reprove Penkov's theorem using the more modern machinery and extend it to general pre-superstacks.
\begin{theorem} \label[theorem]{Crys}
    Let $X$ be a pre-superstack, then the category of crystals on $X$ is equivalent to the category of crystals on $X_{\mathsf{bos}}$. 
\end{theorem}
\begin{proof}
    We have the following commutative diagram 
    \begin{center}
        \begin{tikzcd}
            X  \arrow[r] & X_{\mathsf{dR}} \\
            X_{\mathsf{bos}} \arrow[u] \arrow[r] & X_{\mathsf{dR}} \arrow[u, "\wr", swap]
        \end{tikzcd}
    \end{center}
    the right arrow is the identity. Therefore, we have an equivalence 
    \begin{align}
        \mathsf{QCoh}(X_{\mathsf{dR}}) \simeq  \mathsf{QCoh}((X_{\mathsf{bos} })_{\mathsf{dR}})
    \end{align}.
\end{proof}
\begin{remark}
    We note that the category $\mathsf{QCoh}$ is understood as quasi-coherent sheaves of supermodules, which when evaluated over a bosonic scheme is a $\mathbb{Z}_2$-graded module. Since the isomorphism is given by the identity, this also means it is compatible with truncations. So, we also have the equivalence for the underlying 1-categories, i.e. 
    \begin{align}
        \mathsf{Crys}(X)^\heartsuit \simeq  \mathsf{Crys}(X_{\mathsf{bos}})
    \end{align}
\end{remark}
The rest of this section is dedicated to proving the equivalence of D-modules and crystals for smooth Deligne-Mumford stacks, which justifies our replacement of the classical D-module description.
\begin{corollary}
        Let $\mathsf{sAff}$ denote the category of affine superschemes, then 
    \begin{align}
        \mathsf{Crys}(X) \longrightarrow  \underset{(\mathsf{Spec}(A)\rightarrow X)^{\mathsf{op}}}{\mathsf{lim}} \mathsf{Crys}(\mathsf{Spec}(A))
    \end{align}
    is an equivalence.
\end{corollary}

\begin{proof}
    $\mathsf{QCoh}$ is defined as the right Kan extension of $\mathsf{sAff}^\mathsf{op} \rightarrow \mathsf{LPr}$ and takes colimits in $\mathsf{PresStk}$ to limits in $\mathsf{LPr}$. The claim now follows from \cref{KanPrestk} and \cref{IncKan}.
\end{proof}
For later use, we will need a cofinality result coming from derived algebraic geometry. In particular, we need to prove that for the category of crystals the category of derived schemes of almost finite type, which is the corresponding derived notion of being finite type, will be cofinal. 
\begin{definition}
    Let $X$ be a derived pre-superstack. It is called locally of almost finite type if 
    \begin{enumerate}
        \item $X$ is convergent, i.e., for $S \in \mathsf{dsAff}$, the natural map
        \[
        \mathsf{Map}(S,X) \longrightarrow\underset{n\rightarrow \infty}{\mathsf{lim}} \mathsf{Map}(\tau_{\leq n}S,X)
        \]
        is an isomorphism.
        \item For every $n$, the restriction $X_{\leq n} \coloneqq X|_{\mathsf{dsAff}_{\leq n}}$ lies in $\mathsf{dPresStk}_{\leq n}$
    \end{enumerate}
\end{definition}
\begin{proposition}[{\cite[Proposition 1.3.3.]{gaitsgory2023crystals}}]
    Assume $Z \in \mathsf{dPresStk}^{\mathsf{laft}}$. Then $Z_{\mathsf{dR}}\in \mathsf{dPresStk}^{\mathsf{laft}}$ and $Z_\mathsf{dR} \in \tau_{\leq 0} \mathsf{dPresStk}$. 
\end{proposition}
\begin{proof}
    We need to verify that it is convergent, but 
    \begin{align}
         \mathsf{Map}(S,Z) \longrightarrow\underset{n\rightarrow \infty}{\mathsf{lim}} \mathsf{Map}(\tau_{\leq n}S,Z)
    \end{align}
    is an equivalence by assumption for all affine derived schemes, thus in particular, it is an equivalence for reduced affine schemes. To show that it $Z_\mathsf{dR,\leq n}$ commutes with filtered colimits, we use that $Z$ commutes with filtered colimits, thus it suffices to show 
    \begin{align}
        \mathsf{dsAff} &\longrightarrow \mathsf{dsAff} \\
        S &\longmapsto(\pi_0S)_{\mathsf{br}}
    \end{align}
    commutes with filtered colimits, which follows from \cref{limitscolimits}. To prove the second point we need to prove that 
    \begin{align}
        \underset{S \in \mathsf{dsAff}_{/X}}{\mathsf{colim}} S_\mathsf{dR} 
    \end{align}
    is discrete, for which we will reduce to the affine case. Since $X_\mathsf{dR} \in \mathsf{dPresStk}^{\mathsf{laft}}$
    \begin{align}
        \mathsf{Aff}^\mathsf{ft} \longrightarrow \mathsf{Aff}
    \end{align}
    is cofinal. The category $\tau_{\leq 0}\mathsf{dPresStk}$ is closed under colimits, thus we may assume that $Z$ is a discrete finite type scheme. We now prove that $Z_\mathsf{dR}$ is discrete. Consider an embedding of $Z \rightarrow X$ into a smooth affine finite type scheme $X$. We denote $Y\coloneqq X_Z^\wedge$ and since $X_\mathsf{dR}\rightarrow Y_\mathsf{dR}$ is an isomorphism it suffices to check that $Y_\mathsf{dR}$ is classical. The map
    \begin{align}
        Y^{Y\times_{Y_\mathsf{dR}}[-]} \rightarrow Y
    \end{align}
    is an effective epimorphism combining this with the fact that $\tau_{\leq 0}\mathsf{dPresStk}$ is closed under colimits, we only need to check each term is a discrete scheme. As every term is the formal completion along the embedded diagonal the claim follows.
\end{proof}
\begin{lemma}[{\cite[Corollary 1.3.7.]{gaitsgory2023crystals}}]\label[lemma]{finitetypecofinal}
    Let $\mathcal{C} \subset \mathcal{D}$ be any of the following categories
    \begin{align*}
        \mathsf{sAff}^{\mathsf{ft}}, \mathsf{sAff}, \mathsf{dsAff}^\mathsf{aft}, \mathsf{dsAff}
    \end{align*}
    and $X \in \mathsf{dPresStk}^{\mathsf{laft}}$ then the functor 
    \begin{align}
        \mathcal{C}_{/X_\mathsf{dR}} \longrightarrow \mathcal{D}_{/X_\mathsf{dR}}
    \end{align}
    is cofinal.
\end{lemma}
\begin{proof}
    It suffices to prove this for the composition
    \begin{align}
        \mathsf{sAff}^{\mathsf{ft}} \longrightarrow \mathsf{sAff} \longrightarrow \mathsf{dsAff}
    \end{align}
    The left arrow follows from the fact that $X_\mathsf{dR} \in \mathsf{dPresStk}^\mathsf{laft}$ and the second arrow follows from $X_\mathsf{dR}\in \tau_{\leq 0}\mathsf{dPresStk}$.
\end{proof}
We are now ready to prove the recognition theorem between crystals and $D$-modules. We start by proving that the adjunction $p_!\dashv p^!: \mathsf{QCoh}(X)\rightarrow \mathsf{QCoh}(X_\mathsf{dR})$ is monadic use this and a local computation to show that crystals are equivalent to $D$-modules  
\begin{theorem}[{\cite[Proposition 8.30.]{scholze2025six}}]
    Let $X$ be a smooth superscheme, then the morphism $X\rightarrow X_\mathsf{dR}$ admits $!$-functors and the $!$-functors form a monadic adjunction.
\end{theorem}
\begin{proof}
   We will reduce the general case to the case of $X= \mathbb{A}^{1|1}$ for a finite group $G$ for which we will prove that $!$-functors exist and then verify the conditions for the Barr-Beck-Lurie theorem. For this theorem, we will also need base change in the setting of derived algebraic geometry, thus from now on we consider all objects inside of the category of derived pre-superstacks. The question is local on $X$. Thus, we may assume $X$ is given by some $\mathsf{Spec}(A)$, which admits an étale map to $\mathbb{A}^{n|m}$. The morphism $X \rightarrow X_\mathsf{dR}$ arises out of the base change diagram
   \begin{center}
       \begin{tikzcd}
           X \arrow[r] \arrow[d] & X_\mathsf{dR} \arrow[d] \\
           \mathbb{A}^{n|m} \arrow[r] & \mathbb{A}^{n|m}_{\mathsf{dR}},
       \end{tikzcd}
   \end{center}
   so we can reduce to $X= \mathbb{A}^{n|m}$ and by taking products we can reduce to the case $\mathbb{A}^1$. Moreover, we can also assume $X = \mathbb{P}^{1|1}$. So from now on, we may assume that $X$ is smooth and proper. By cofinality from \cref{finitetypecofinal}, we need to prove that for any derived $k$-scheme of almost finite type $Y$ equipped with a morphism $Y \rightarrow\mathbb{A}^{1|1}_\mathsf{dR}$ the fiber product $X \times_{X_\mathsf{dR}} Y \rightarrow Y$ admits $!$-functors, Grothendieck duality holds and is surjective. The question of $!$-functors is local on the source and target, so we may assume that $Y \simeq \mathsf{Spec}(B)$ and the morphism $Y \rightarrow X_\mathsf{dR}$ arises from a morphism $Y_{\mathsf{red}} \rightarrow X$. The graph of the morphism $Y_\mathsf{red} \rightarrow X$ gives a closed subset $Z \subset X \times_k Y$, which implies
   \begin{align}
       X \times_{X_\mathsf{dR}} Y \simeq (X \times_k Y)^\wedge_Z,
   \end{align}
    since $(X \times_k Y)^\wedge_Z$ is the subfunctor of maps that set-theoretically factor through $Z$. Using \cref{recollement} we obtain that 
    \begin{align}
       \hat{p}: X \times_{X_\mathsf{dR}} Y \simeq (X \times_k Y)^\wedge_Z \longrightarrow X \times_k Y
    \end{align}
    admits $!$-functors. Furthermore, $\hat{p}^* \simeq \hat{p}^!$. To prove that we satisfy Grothendieck duality, we use the commutative square
    \begin{center}
        \begin{tikzcd}
            X \times_k Y \arrow[r] \arrow[d] & Y \arrow[d]\\
            X \arrow[r] & \mathsf{Spec}(k)
        \end{tikzcd}
    \end{center}
     and observe that $X \rightarrow \mathsf{Spec}(k)$ is smooth and proper by assumption, which is stable by pullback. It now follows from \cite[Relative Grothendieck Duality]{bruzzo2023notes} that $p$ satisfies Grothendieck duality and it is easy to see that $X \rightarrow X_{\mathsf{dR}}$ is surjective. \\
     To verify the Barr-Beck-Lurie criteria, we observe that conservativity of $p^!$ follows from Grothendieck duality and $p^!$ preserves all colimits and limits. In particular, it preserves split geometric realizations.
   \end{proof}
To elaborate on what was done, we have an adjunction
\begin{align}
   p^!\dashv p_!: \mathsf{QCoh}(X_{\mathsf{dR}}) \adjoint \mathsf{QCoh}(X).
\end{align}
 We denote $\mathsf{Diff}_X\coloneqq p^! \circ p_!$ and due to monadicity, we obtain the equivalence 
\begin{align}
    \mathsf{Crys}(X) \simeq \mathsf{Mod}_{\mathsf{Diff}_X}(\mathsf{QCoh}(X)).
\end{align}
 Given their nature, we will also endow the adjunction functors with new names as follows $\mathsf{oblv}\coloneqq p^! $ and $\mathsf{ind} \coloneqq p_!$. They are the forgetful functor, which forgets that we are a crystal and the induction functor, to turn a quasi-coherent sheaf into a crystal. $\mathsf{Diff}_{X}$ plays the role of the sheaf of differential operators. In the smooth case, this is a theorem as we will now prove using a local computation on $\mathbb{A}^{1|1}$.
\begin{example}[$\mathbb{A}^{1|1}$, {\cite{Cai2024DModules}}] \label[example]{DModisCrys}
    Let $k$ be an algebraically closed field of characteristic zero. We define the functor 
    \begin{align}
        \begin{split}
            \hat{\mathbf{0}}: \mathsf{sSch} &\longrightarrow \mathsf{sSch} \\
                    \mathsf{Spec}(A) &\longmapsto \mathsf{Nil}(A)
        \end{split}
    \end{align}
    This functor arises from the following situation. Consider the pullback diagram taken in the category of $\mathcal{P}(\mathsf{sAff})$
    \begin{center}
        \begin{tikzcd}
            \mathsf{Spec}(A) \arrow[drr, bend left = 20] \arrow[ddr, bend right = 20] \arrow[dr, dashed] & & \\
            & X \arrow[r] \arrow[d] \arrow[dr, phantom,"\lrcorner", very near start] & \mathbb{A}^{1|1} \arrow[d] \\
            &* \arrow[r] & \mathbb{A}^{1|1}_{\mathsf{dR}}
        \end{tikzcd}
    \end{center}
    we dualize this diagram
        \begin{center}
        \begin{tikzcd}
            A   & & \\
            & \mathcal{O}(X) \arrow[ul, dashed] \arrow[dr, phantom,"\lrcorner", very near start] & k[t|\theta] \arrow[l] \arrow[ull, bend right = 20 ] \\
            &k \arrow[uul, bend left = 20] \arrow[u] & k[t] \arrow[u, "\mathsf{inc}", swap] \arrow[l, "t \mapsto 0"]
        \end{tikzcd}
    \end{center}
    We observe that to be a pushout the composites 
    \begin{align}
        k[t] \rightarrow k \rightarrow A = k[t] \rightarrow k[t|\theta] \rightarrow A.
    \end{align}
    have to agree. Since, the left hand side factors through $A_{\mathsf{br}}$, so does the right hand side. It follows that $(t|\theta) \mapsto (r|s_1) \in \mathsf{Nil}(R) = \mathsf{Nil}(A_0) \oplus A_1$. We recall, that $ \mathsf{Spec}(A) \rightarrow X \in X(\mathsf{Spec}(A))$. So the corresponding functor can be identified as 
    \begin{align}
        \mathsf{Hom}(\mathcal{O}(X), A) \simeq  \mathsf{Nil}(A).
    \end{align}
    We conclude the pre-superstack is given by the assignment
    \begin{align}
        \begin{split}
                    X: \mathsf{sAff}_{\mathsf{op}} &\longrightarrow \mathsf{Gpd} \\
            \mathsf{Spec}(A) &\longmapsto \mathsf{Nil}(A). 
        \end{split}
    \end{align}
    We recognize that $X = \hat{\mathbf{0}}$. We can now analyze the functor of points of 
    \begin{align}
        \begin{split}
             \mathbb{A}^{1|1}/\hat{\mathbf{0}}: \mathsf{CAlg}_k &\longrightarrow \mathsf{Gpd} \\
        A &\longmapsto \mathsf{Hom}(k[t|\theta],A)/\mathsf{Nil}(A),
        \end{split}
    \end{align}
    where we observe
    \begin{align}
        \mathsf{Hom}(k[t|\theta],A)/\mathsf{Nil}(A)= A/\mathsf{Nil}(A) = A_{\mathsf{br}}
    \end{align}
    thus $\mathbb{A}^{1|1}/\hat{\mathbf{0}}$ isomorphic to $\mathbb{A}^{1|1}_{\mathsf{dR}}$. As a consequence, we obtain 
    \begin{align}
        \mathsf{QCoh}(\mathbb{A}^{1|1}_{\mathsf{dR}}) \simeq  \mathsf{QCoh}^{\hat{\mathbf{0}}}(\mathbb{A}^{1|1}).
    \end{align}
    Let $R$ be a superring. An $R$-action is the same thing as being an $k[R]$-comodule, or in this case an $\mathcal{O}(\mathbf{\hat{0}})$-comodule or modules with an $\mathcal{O}(\mathbf{\hat{0}})^\vee$-action. We compute 
    \begin{align}
        \mathcal{O}(\mathbf{\hat{0}}) = \mathcal{O}(\underset{n}{\mathsf{colim}\,\mathsf{Spec}(k[t|\theta]/t^n)}) = \underset{n}{\mathsf{lim}} k[t|\theta]/t^n = k \llbracket t|\theta \rrbracket = k\llbracket t\rrbracket \otimes k[\theta].
    \end{align}
    the dual vector space of this is $k[\frac{\partial_t^n}{n!}| \partial_\theta]$. So for $ M \in \mathsf{QCoh}^{\hat{\mathbf{0}}}(\mathbb{A}^{1|1})$ it needs to be a module with respect to  $k[t|\theta] \otimes k[\frac{\partial_t^n}{n!}| \partial_\theta]$. This is the algebra of differential operators, which we will denote by $W$. Therefore, 
    \begin{align}
        \mathsf{Crys}(\mathbb{A}^{1|1}) \simeq  \mathsf{QCoh}(\mathbb{A}^{1|1}_{\mathsf{dR}}) \simeq  \mathsf{Mod}(W).
    \end{align}
\begin{remark}
    This is all done on the level of derived categories, but this also holds for the discrete objects, i.e. 
    \begin{align}
        \mathsf{QCoh}(\mathbb{A}^{1|1}_{\mathsf{dR}})^\heartsuit \simeq  \mathsf{Mod}(W)^{\heartsuit}.
    \end{align}
\end{remark}
We can now also compute, what the action of the monad $\mathsf{Diff}_{\mathbb{A}^{1|1}} \coloneqq \mathsf{oblv} \circ \mathsf{ind}$ does. Let $M \in \mathsf{Crys}(\mathbb{A}^{1|1})$, and we identify 
\begin{align}
  \mathsf{ind} \dashv \mathsf{oblv}:  \mathsf{Mod}(k[t|\theta]) \adjoint \mathsf{Mod}(W)
\end{align}
then the monad acts on $M$ by 
\begin{align}
    \mathsf{Diff}_{\mathbb{A}^{1|1}}(M) \longrightarrow M.
\end{align}
In our example we can see that the induction functor is $W \otimes (-)$ and $\mathsf{oblv}$ is the forgetful functor. Thus, applying $(\mathsf{oblv} \circ \mathsf{ind})(M)$ yields the action 
\begin{align}
    W \otimes M \rightarrow M,
\end{align}
by the natural action of $W$ on the $k[t|\theta]$-module $M$. We conclude that this monad exactly encodes an action of the sheaf of differential operators, which is the classic definition of a $D$-module. 
\end{example}
This example extends to arbitrary $\mathbb{A}^{n|m}$ and thus to all smooth superschemes, by étale descent. Thus, we have the theorem
\begin{theorem} 
Let $X$ be a smooth superscheme, then
    \begin{align}
        \mathsf{Crys}(X) \simeq  D\mathsf{Mod}(X).
    \end{align}
\end{theorem}
Using this, we can make the following identifications. Let $X$ be an Artin superstack and $U \to X$ be a smooth presentation of $X$ by an algebraic superspace. Then we have the equivalences
        \begin{align}
      \mathsf{Crys}(X) \simeq  \underset{\mathsf{Sch}_{/X}^\mathsf{op}}{\mathsf{lim}}\mathsf{Crys}(S) \simeq \underset{\mathsf{Sch}_{/U}^\mathsf{op}}{\mathsf{lim}}\mathsf{Crys}(S) \simeq \underset{\mathsf{Et}(U)^\mathsf{op}}{\mathsf{lim}}\mathsf{Crys}(S) \simeq 
      \underset{\mathsf{Et}(U)^\mathsf{op}}{\mathsf{lim}}\mathsf{DMod}(S)
    \end{align}
and one may give the following definition:
\begin{definition}
    We define the category of $D$-modules on a smooth Artin superstack to be
    \begin{align}
        \mathsf{DMod}(X) \coloneqq \mathsf{Crys}(X).
    \end{align}
\end{definition}
Using the geometric perspective of de Rham cohomology, we also obtain a spectral sequence free proof of the equivalence of de Rham cohomology of the supermanifold and de Rham cohomology of the underlying manifold.
\begin{theorem}\label[theorem]{deRhamisCoherent}
    Let $X$ be a smooth superscheme then 
    \begin{align}
      R^\bullet\Gamma(X_{\mathsf{dR}}, \mathcal{O}_{X_{\mathsf{dR}}})  \simeq  \mathbb{H}^\bullet(X,\Omega_X^\bullet) = H_{\mathsf{dR}}(X, \mathbb{C})
    \end{align}
\end{theorem}
\begin{proof}
    Recall that for a superscheme $X$ coherent cohomology is given by 
    \begin{align}
        \mathsf{Map}(\mathcal{O}_X, \mathcal{F}) \simeq  \mathsf{R}\Gamma(X,\mathcal{F}).
    \end{align}
    We wish to compute $\mathsf{Map}_{\mathsf{QCoh}(X_\mathsf{dR})}(\mathcal{O}_{X_\mathsf{dR}},\mathcal{O}_{X_\mathsf{dR}})$, which we view as a $\mathbb{C}$-linear spectrum. Using \cref{DModisCrys} we obtain the equivalence 
    \begin{align}
        \mathsf{Map}_{\mathsf{QCoh}(X_\mathsf{dR})}(\mathcal{O}_{X_\mathsf{dR}},\mathcal{O}_{X_\mathsf{dR}}) \simeq  \mathsf{Map}_{\mathsf{DMod}(X)}(\mathcal{O}_{X},\mathcal{O}_{X}) \simeq  \Gamma(X,\mathsf{R}\mathcal{H}om_{\mathcal{D}_X}(\mathcal{O}_X,\mathcal{O}_X)).
    \end{align}
    We can resolve $\mathcal{O}_X$ using the spencer complex given by 
    \begin{align}
  ... \rightarrow \wedge^2\mathcal{T}_X \rightarrow \mathcal{T}_X \rightarrow \mathcal{O}_X
    \end{align}
    Thus, we have the equivalence of $\mathbb{C}$-linear spectra:
    \begin{align}
        \mathsf{R}^\bullet\mathcal{H}om(\mathcal{O}_X,\mathcal{O}_X) \simeq \mathcal{H}om_{\mathcal{O}_X}(\wedge^\bullet \mathcal{T}_X,\mathcal{O}_X) \simeq ( \wedge^\bullet \Omega_X, d_\mathsf{dR}).
    \end{align}
    However, using the groupoid presentation of $X_\mathsf{bos} \rightarrow X_\mathsf{dR}$, we also obtain that $\mathsf{R}^\bullet\Gamma(X_\mathsf{dR},\mathcal{O}_{X_\mathsf{dR}}) \simeq \wedge^\bullet (\Omega_{X_\mathsf{bos}},d_\mathsf{dR})$, as the coherent cohomology is invariant under presentation, and the claim follows.
\end{proof}
\begin{theorem}\label[theorem]{quotientDmod}
    Let $X$ be a smooth superscheme and let $G$ be an algebraic supergroup and consider the stack quotient $[X/G]$, then we have the equivalence
    \begin{align}
        \mathsf{Crys}([X/G]) \simeq \mathsf{Crys}([X_\mathsf{bos}/G_\mathsf{bos}]).
    \end{align}
\end{theorem}
\begin{proof}
 Let $X$ be a smooth superscheme and let $G$ be an algebraic supergroup and consider the stack quotient $[X/G]$, 
 By \cref{limitscolimits} the de Rham functor commutes with colimits and a quotient is a coequalizer. Combining these two facts, we obtain the following isomorphism 
    \begin{align}
        \mathsf{Crys}([X/G]) \simeq  \mathsf{QCoh}([X/G]_{\mathsf{dR}}) \simeq  \mathsf{QCoh}([X_\mathsf{dR}/G_\mathsf{dR}]) \simeq  \mathsf{Crys}([X_{\mathsf{red}}/{G}_{\mathsf{red}}])
    \end{align}
\end{proof}
As a corollary, we can apply this to the flag supervariety of a quasi-reductive algebraic supergroup and see the classical Beilinson-Bernstein localization theorem fails.
\begin{corollary}
   Let $G$ be a quasi-reductive algebraic supergroup and fix a Borel subgroup. Furthermore, consider the flag supervariety $G/B$, then 
   \begin{align}
       \mathsf{Crys}(G/B) \simeq  \mathsf{Crys}(G_{\mathsf{red}}/{B}_{\mathsf{red}}) \simeq  \mathsf{Mod}(\mathcal{U}\mathfrak{g}_{\mathsf{red}}/\mathsf{ker}(\chi))
   \end{align}
\end{corollary}
\begin{proof}
    First, the action of $B$ is free on $G$, thus $[G/B] \simeq  G/B$. Now, it is the same chain of equivalences as \cref{quotientDmod} and the last equivalence follows from the classical Beilinson-Bernstein theorem for a complex reductive algebraic (ordinary) group.
\end{proof}
\begin{example}[Weakly equivariant D-modules]
Let $X$ be a smooth superscheme and $G$ an algebraic supergroup. $X_{\mathsf{dR}}$ inherits an action of $G$, so we can also consider the stack quotient $[X_\mathsf{dR}/G]$. We obtain
\begin{align}
    \mathsf{QCoh}([X_\mathsf{dR}/G]) \simeq  \mathsf{QCoh}^G(X_\mathsf{dR}) \simeq \mathsf{Crys}^G(X).
\end{align}
Objects in this category are known as \textit{weakly equivariant} $D$-modules.
\end{example}
\begin{example}[Affine Grassmannians for Supergroups]
    Let $G$ be a quasi-reductive algebraic supergroup. Following the definition of the formal superdisk in \cite{maxwell2024neveu}, we define the loop supergroup
    \begin{align}
        \begin{split}
            G((\!( t|\theta)\!)): \mathsf{sCAlg} &\longrightarrow \mathsf{Gpd} \\
           A &\longmapsto G(A(\!( t|\theta)\!))
        \end{split}
    \end{align}
    and
        \begin{align}
        \begin{split}
            G([\![ t|\theta]\!]): \mathsf{sCAlg} &\longrightarrow \mathsf{Gpd} \\
           A &\longmapsto G(A[\![ t|\theta]\!])
        \end{split}
    \end{align}
    We can now also define the quotient of these two functors 
    \begin{align}
        \begin{split}
            \mathsf{Gr}_G: \mathsf{sCAlg} &\longrightarrow \mathsf{Gpd} \\
           A &\longmapsto G(A(\!( t|\theta)\!))/G(A[\![ t|\theta]\!]),
        \end{split}
    \end{align}
    which is a prestack and we can fppf sheafify this to obtain the affine Grassmannian, which is inessential since the category of quasi-coherent sheaves is insensitive to the stack structure. Applying the de Rham functor we obtain
    \begin{align}
        \begin{split}
            (\mathsf{Gr}_G)_{\mathsf{dR}}: \mathsf{sCAlg} &\longrightarrow \mathsf{Ani} \\
           A &\longmapsto G((A(\!( t|\theta)\!))_{\mathsf{br}})/G((A[\![ t|\theta]\!])_{\mathsf{br}}) \simeq G(A_{\mathsf{br}}(\!(t)\!))/G(A_{\mathsf{br}}[\![t]\!]).
        \end{split}
    \end{align}
    Thus, we have 
    \begin{align}
        \mathsf{Crys}(\mathsf{Gr}_G) \simeq  \mathsf{Crys}(\mathsf{Gr}_{G_\mathsf{red}})
    \end{align}
\end{example}
\subsection{Constructible Sheaves}
We give a refinement Simpson's notion of the Betti stack to the setting of constructible sheaves.  The Betti stack is the geometrization of singular cohomology and its quasi-coherent sheaves encode local systems. It represents the topological side of the non-abelian Hodge correspondence, thus, by construction, it is not sensitive to any complex or supergeometric information, thus for the remainder of the section we write $?\mathsf{Aff}$, which stands for any of the following categories:
\begin{align}
    \mathsf{Aff}, \ \ \mathsf{dAff}, \ \ \mathsf{sAff}, \ \ \mathsf{dsAff}, \ \ \mathsf{DAff}.
\end{align}
The Betti stack is informed by the monodromy equivalence, which states that, for a presentable category $\mathcal{C}$ and a locally contractible space $X$, there is a fully faithful functor 
\begin{align}
    \mathsf{Fun}(\Pi_\infty(X), \mathcal{C}) \longrightarrow \mathsf{Fun}(\mathsf{Open}(X)^\mathsf{op},\mathcal{C}),
\end{align}
whose essential image are the locally constant sheaves. The monodromy equivalence also has a stratified refinement, where the constructible sheaves take the place of the locally constant ones, called the exodromy equivalence, which states that for a poset $P$ and $X$ a conically $P$-stratified topological space, there exists a fully faithful functor 
\begin{align}
    \mathsf{Fun}(\mathsf{Exit}(X,P),\mathcal{C}) \longrightarrow \mathsf{Fun}(\mathsf{Open}(X)^\mathsf{op},\mathcal{C}),
\end{align}
whose image are the $P$-constructible sheaves, where the specific statement was taken from \cite[\S 7.2.]{haine2021homotopy}. The obvious obstruction to obtaining a notion of $P$-stratified Betti stack is that the so-called exit-path category of $X$, denoted by $\mathsf{Exit}(X,P)$, is not an anima, but a category. However, there is a notion of presheaf of categories, also called lax prestacks, which to the authors knowledge was pioneered in \cite{raskin2015chiral}, which lets us refine the notion of Betti stack to the stratified setting. Let us start by defining what a lax prestack is.
\begin{definition}
    A \textit{lax prestack}, is an element in $\mathsf{Fun}(?\mathsf{Aff}^\mathsf{op},\mathsf{Cat})$.
\end{definition}
    Since $\mathsf{QCoh}$, is a sheaf of categories, so in particular a lax prestack. Let $\mathsf{Nat}(-,-)$ denote the category of natural transformations then we may define the category of quasi-coherent sheaves on a lax prestack $X$ as follows:
\begin{definition}
    Let $X$ be a lax prestack then we define the category of \textit{quasi-coherent sheaves} on $X$ to be 
    \begin{align}
        \mathsf{QCoh}^{\mathrm{lax}}(X) \coloneqq \mathsf{Nat}(X, \mathsf{QCoh}).
    \end{align}
\end{definition}
This definition is rather unworkable, thus let us introduce the necessary concepts to give a more explicit description in terms of a limit. 
Let $\mathcal{C}$ be a category then $\mathsf{Tw}(\mathcal{C})$ is called the \textit{twisted arrow category}, whose objects are morphisms $f: x \to y $ in $\mathcal{C}$ and morphisms are given by commutative squares of the form 
    \begin{center}
        \begin{tikzcd}
            x \arrow[d, "f"] \arrow[r]  &  x' \arrow[d, "f'"] \\
            y  & \arrow[l] y'
        \end{tikzcd},
    \end{center}
where a precise definition in terms of a model of $\infty$-categories is given in \cite[Definition 2.2.]{gepner2017lax}.
Let us now also admit the notion of an end.
\begin{definition}[cf. {\cite{haugseng2022co}}]
   Let $p: \mathsf{Tw}(\mathcal{C})^\mathsf{op} \to \mathcal{C}^\mathsf{op} \times \mathcal{C}$ be the functor that takes the source and target of morphism. We define the \textit{end} of a functor $F:\mathcal{C}^\mathsf{op} \times \mathcal{C} \to \mathcal{D}$ to be the limit of the composite 
   \begin{align}
       F \circ p: (\mathsf{Tw}(\mathcal{C}))^\mathsf{op} \longrightarrow \mathcal{D}.
   \end{align}
\end{definition}
It also turns out that the category of natural transformations has a description in terms of ends. 
\begin{proposition}[{\cite[Proposition 5.1.]{gepner2017lax}}] \label[proposition]{natural-transformations-are-ends}
   For functors $F,G: \mathcal{C} \to \mathcal{D}$ we have 
    \begin{align}
        \mathsf{Nat}_{\mathcal{C},\mathcal{D}}(F,G) \simeq \underset{X \to Y \in (\mathsf{Tw}(X))^\mathsf{op}}{\mathsf{lim}}\mathsf{Map}(-,-) \circ (F^\mathsf{op} \times G) \circ p
    \end{align}
\end{proposition}
Finally, we can give the aforementioned explicit description of the category of quasi-coherent sheaves on lax prestacks. 
\begin{corollary} \label[corollary]{twisted-arrow-limit}
    Let $X$ be a lax prestack then we can identify the category of quasi-coherent sheaves on $X$ as, 
    \begin{align}
        \mathsf{QCoh}^{\mathrm{lax}}(X) \simeq \underset{(S' \to S \in \mathsf{Tw}(?\mathsf{Aff}))^\mathsf{op}}{\mathsf{lim}} \mathsf{Fun}(X(S), \mathsf{QCoh}(S'))
    \end{align}
\end{corollary}
\begin{proof}
    This is an immediate consequence of \cref{natural-transformations-are-ends}, by setting $F=X$ and $Y= \mathsf{QCoh}$.
\end{proof}
Furthermore, the notion of quasi-coherent sheaves of lax prestacks coincides with the usual notion of category of quasi-coherent sheaves on prestacks. 
\begin{lemma} \label[lemma]{lax-is-the-same-as-normal}
    The following diagram of categories commutes 
    \begin{center}
        \begin{tikzcd}[column sep = huge, row sep = large]
            \mathsf{Fun}(?\mathsf{Aff}^\mathsf{op},\mathsf{Ani}) \arrow[d, hook] \arrow[r, "\mathsf{QCoh}(-)"] & \mathsf{LPr} \\
            \mathsf{Fun}(?\mathsf{Aff}^\mathsf{op},\mathsf{Cat}) \arrow[ur, "\mathsf{QCoh}^\mathsf{lax}(-)"']
        \end{tikzcd}
    \end{center}
\end{lemma}
\begin{proof}
    Let $X$ be a prestack and recall that, by definition we have
    \begin{align}
        \mathsf{QCoh}^\mathrm{lax}(X) \coloneqq \mathsf{Nat}(X,\mathsf{QCoh}).
    \end{align}
    Since $X$ is a small colimit over representables, i.e. $X \simeq \underset{S \to X \in ?\mathsf{Aff}_{/X}}{\mathsf{colim}}S$, and for $\mathcal{C,D} \in \mathsf{Cat}$ it holds that $ \mathsf{Map}(-,-): \mathcal{C}^\mathsf{op} \times \mathcal{C} \to \mathcal{D}$ commutes with limits in both variables, we have
\begin{align}
    \mathsf{Nat}(X,\mathsf{QCoh}) \simeq \mathsf{Nat}(\underset{S \to X \in ?\mathsf{Aff}_{/X}}{\mathsf{colim}}S, \mathsf{QCoh}) \simeq \underset{S \to X \in (?\mathsf{Aff}_{/X})^\mathsf{op}}{\mathsf{lim}} \mathsf{Nat}(S,\mathsf{QCoh}),
\end{align}
and hence, by the Yoneda lemma, we can identify $\mathsf{Nat}(S,\mathsf{QCoh}) \simeq \mathsf{QCoh}(S)$ and hence we obtain 
\begin{align}
    \mathsf{Nat}(X,\mathsf{QCoh}) \simeq \underset{S \to X \in (?\mathsf{Aff}_{/X})^\mathsf{op}}{\mathsf{lim}} \mathsf{QCoh}(S),
\end{align}
as wished.
\end{proof}
Thus, from now on we will drop the superscript 'lax' in the upcoming discussion. Let $\mathsf{Str}_P$ denote the category of $P$-stratified homotopy types, which is equivalent to the full category of categories over a poset $P$ $\mathsf{Cat}_{/P}$, such that the functor $\mathcal{C} \to P$ is conservative. For more details on the theory of stratified homotopy theory, we recommend \cite{haine2021homotopy} and \cite{barwick2018exodromy}.
Let $\mathsf{Str}_P^{\mathrm{ex}}$ be the full subcategory of $\mathsf{Str}_P$, spanned by those stratified homotopy types for whom the exit-path category $\mathsf{Exit}(X,P)$ is an $\infty$-category. There exists an obvious adjunction 
\begin{align}
    \pi^*: \mathsf{Str}_P^\mathrm{ex} \adjoint\mathsf{Fun}(?\mathsf{Aff}^\mathsf{op},\mathsf{Str_P^\mathrm{ex}}): \pi_* 
\end{align}
\begin{definition}
    Let $X \in \mathsf{Str}_P^\mathrm{ex}$ then we define the \textit{stratified Betti transmutation} to be the functor 
    \begin{align}
    \begin{split}
        \pi^*: \mathsf{Str}_P^\mathrm{ex} &\longrightarrow \mathsf{Fun}(?\mathsf{Aff}^\mathsf{op},\mathsf{Str_P^\mathrm{ex}}) \\
       X &\longmapsto X_{B_P}\coloneqq \mathsf{Exit}(X,P),
    \end{split}
    \end{align}
    which assigns to a $P$-stratified homotopy type, modeled by its exit path category $\mathsf{Exit}(X,P)$, the constant presheaf associated to it. We will call $X_{B_P}$ the $P$\textit{-stratified Betti stack} of $X$.
\end{definition}

\begin{theorem} \label[theorem]{stratifiedbetti}
    Let $X \in \mathsf{Str}_P^\mathrm{ex}$ then    
    \begin{align}
        \mathsf{QCoh}(X_{B_P} \times \mathsf{Spec}(A)) \simeq \mathsf{Fun}(\mathsf{Exit}(X,P), \mathsf{Mod}_A).
    \end{align}
\end{theorem}
\begin{proof}
    By definition, we have 
    \begin{align}
        \mathsf{QCoh}(X_{B_P} \times \mathsf{Spec}(A)) \simeq
        \underset{(S' \to S \in \mathsf{Tw}(?\mathsf{Aff}))^\mathsf{op}}{\mathsf{lim}} \mathsf{Fun}((X_{B_P} \times \mathsf{Spec}(A))(S), \mathsf{QCoh}(S')).
    \end{align}
    Since products of functors are computed pointwise, we can give the following simplification
    \begin{align}
        \underset{(S' \to S \in \mathsf{Tw}(?\mathsf{Aff}))^\mathsf{op}}{\mathsf{lim}} \mathsf{Fun}(X_{B_P} (S) \times (\mathsf{Spec}(A))(S), \mathsf{QCoh}(S')),
    \end{align}
    thus using the fact that $\mathsf{Cat}$ is cartesian closed, we can use the internal tensor-hom adjunction, which produces the following expression:
    \begin{align} \label{equation}
        \underset{(S' \to S \in \mathsf{Tw}(?\mathsf{Aff}))^\mathsf{op}}{\mathsf{lim}} \mathsf{Fun}(X_{B_P} (S),\mathsf{Fun}(\mathsf{Spec}(A)(S), \mathsf{QCoh}(S') )).
    \end{align}
    Now, one observes that the stratified Betti stack evaluates to $X_{B_P}(S)\simeq \mathsf{Exit}(X,P)$ as it is the constant presheaf on $\mathsf{Exit}(X,P)$, and using the fact that $\mathsf{Fun}(-,-)$ commutes with limits in the second variable that one can identify \eqref{equation} as:
    \begin{align}
         \mathsf{Fun}(\mathsf{Exit}(X,P), \underset{(S' \to S \in \mathsf{Tw}(?\mathsf{Aff}))^\mathsf{op}}{\mathsf{lim}}\mathsf{Fun}(\mathsf{Spec}(A)(S), \mathsf{QCoh}(S') )).
    \end{align}
    Lastly, one observes that the second factor is the definition of the category of quasi-coherent sheaves on a lax prestack, which coincides with the usual category of quasi-coherent sheaves for a prestack, by \cref{lax-is-the-same-as-normal}. Since $\mathsf{Spec}(A)$ is affine, we also identify $\mathsf{QCoh}(\mathsf{Spec}(A)) \simeq \mathsf{Mod}_A$, thus we conclude that
    \begin{align}
        \mathsf{QCoh}(X_{B_P} \times \mathsf{Spec}(A)) \simeq \mathsf{Fun}(\mathsf{Exit}(X,P), \mathsf{Mod}_A),
    \end{align}
    as desired.
\end{proof}
\begin{remark}
    This proof also naturally upgrades to the symmetric monoidal setting as $\mathsf{Cat}^\otimes = \mathsf{CAlg}(\mathsf{Cat})$, which admits a forgetful functor $U:\mathsf{CAlg}(\mathsf{Cat}) \to \mathsf{Cat}$, so the limit can be computed underlying and the operation of tensor-hom adjunction is compatible with symmetric monoidal structures.
\end{remark}
\begin{remark}
    If we set $P=\mathsf{pt}$ then we recover the equivalence between the Betti stack and locally constant sheaves, where for an alternative proof see \cite{PortaSalaShapes}.
\end{remark}

\subsection{Higgs Sheaves}
Let $X$ be a smooth superscheme of dimension $n|m$ and let $\mathcal{E} \in \mathsf{QCoh}(X)^{\heartsuit,\mathsf{cp}}$. The 1-category of quasi-coherent sheaves is symmetric monoidal, thus it makes sense to define algebra objects inside of this category. Naturally, there also exists the free commutative algebra functor:
\begin{align}
\begin{split}
        \mathsf{Sym}: \mathsf{QCoh}(X)^\heartsuit &\longrightarrow \mathsf{CAlg}(\mathsf{QCoh}(X)^\heartsuit) \\
    \mathcal{E} &\longmapsto \mathsf{Sym}\mathcal{E}.
\end{split}
\end{align}
This means, $\mathsf{Sym}\mathcal{E}$ is an $\mathcal{O}_X$-algebra, which enables us to define the relative spectrum functor. 
\begin{align}
    E \coloneqq \mathsf{Spec}(\mathsf{Sym}(\mathcal{E}^\vee)).
\end{align}
 The functor $\mathsf{Sym}$ is a left adjoint, so it preserves all colimits and in particular, it sends cogroup objects to cogroup objects. So the functor restricts to 
 \begin{align}
     \mathsf{Sym}: \mathsf{CoGrp}(\mathsf{QCoh}(X)^\heartsuit) &\longrightarrow \mathsf{CoGrp}(\mathsf{CAlg}(\mathsf{QCoh}(X)^\heartsuit)).
 \end{align}
 Therefore, composing with the relative spectrum functor 
 \begin{align}
     \mathsf{QCoh}(X)^{\heartsuit,\mathsf{op} }&\xlongrightarrow{\mathsf{Sym}^{\mathsf{op}}} \mathsf{CAlg}(\mathsf{QCoh}(X)^\heartsuit)^{\mathsf{op}} \xlongrightarrow{\mathsf{Spec}} \mathsf{Shv}(\mathsf{sAff})_{/X}
 \end{align}
 equips every total space $\mathsf{Spec}(\mathsf{Sym}\mathcal{E})$ with the trivial group structure.
 \begin{definition}
     Let $X$ be a smooth superscheme and let $TX \coloneqq \mathsf{Spec}.(\mathsf{Sym}\Omega_X)$ The Dolbeault stack is defined to be the quotient stack $X_{\mathsf{Dol}}=[X/\widehat{TX}]$. Where $\widehat{TX}$ is the formal completion along the zero-section.
 \end{definition}
 There also exists a variant of the Dolbeault stack
 \begin{definition}
     The nilpotent Dolbeault stack is defined to be 
     \begin{align}
         X^{\mathsf{nil}}_{\mathsf{Dol}}\coloneqq [X/TX].
     \end{align}
 \end{definition}
 From the definitions we now also obtain the following theorems.
\begin{theorem}
    Let $X$ be a smooth scheme of purely odd dimension $m$. Then we have the following isomorphism: 
    \begin{align}
        X^{\mathsf{nil}}_{\mathsf{Dol}} \cong X_{\mathsf{Dol}}.
    \end{align}
\end{theorem}
\begin{proof}
    This follows by identifying the formal completion of the tangent bundle and the tangent bundle. This question is local so we can reduce to the $\mathbb{A}^{0|1}$ case and extend to all other cases by étale descent. One computes that $\mathcal{O}_{TX} = \mathcal{O}_{\mathbb{A}^{0|1}}[\theta]$ and the zero section is given by the morphism $\theta \mapsto 0$. Therefore, 
    \begin{align}
        \widehat{T\mathbb{A}^{0|1}} = \underset{m \in \mathbb{N}}{\mathsf{colim}} \mathsf{Spec}(\mathcal{O}_{\mathbb{A}^{0|1}}[\theta]/(\theta)^m) \cong \mathsf{Spec}(\mathcal{O}_{\mathbb{A}^{0|1}}[\theta]) \cong T{\mathbb{A}^{0|1}}.
    \end{align}
    Thus, we obtain the desired isomorphism 
    \begin{align}
         X^{\mathsf{nil}}_{\mathsf{Dol}} = [X/TX] \cong [X/\widehat{TX}] \cong X_{\mathsf{Dol}}.
    \end{align}
\end{proof}
This shows that the odd directions only add nilpotent higgs fields.
\begin{theorem}
    Let $X$ be a smooth superstack. Then we have 
    \begin{align}
        \mathsf{QCoh}(X_{\mathsf{Dol}}) \cong \mathsf{Mod}_{\mathsf{Sym}\mathcal{T_X}}(\mathsf{QCoh}(X)) \ \text{and}\ \mathsf{QCoh}(X_{\mathsf{Dol}}^\mathsf{nil}) \cong \mathsf{Mod}_{\widehat{\mathsf{Sym}\mathcal{T_X}}}(\mathsf{QCoh}(X))
    \end{align}
    In particular, $\mathsf{QCoh}(X_{\mathsf{Dol}})^\heartsuit$ is the category of Higgs sheaves on $X$.
\end{theorem}
\begin{proof}
    We note that $\mathsf{Sym}(\mathcal{T}_X) \in \mathsf{CAlg}(\mathsf{QCoh}(X))^\heartsuit$. By definition, we obtain the equivalence 
    \begin{align}
        \mathsf{QCoh}(X_\mathsf{Dol}) \cong \mathsf{Mod}_{\widehat{\mathsf{Sym}\Omega_X}^\vee}(\mathsf{QCoh}(X)).
    \end{align}
    To identify $\widehat{\mathsf{Sym}\Omega_X}^\vee$, we use that $\mathsf{Hom}(-,\mathcal{O}_X)$ sends colimits to limits and $\mathsf{Sym}\mathcal{T}_X \simeq \bigoplus \mathsf{Sym}^n\mathcal{T}_X$. Therefore, we have 
    \begin{align}
        \mathsf{Sym}\mathcal{T}_X^\vee \simeq \prod \mathsf{Hom}(\mathsf{Sym}^n \mathcal{T}_X, \mathcal{O}_X) \simeq \widehat{\mathsf{Sym}\Omega_X}.
    \end{align}
    Similarly, one can obtain that ${\mathsf{Sym}\Omega_X}^\vee \simeq \widehat{\mathsf{Sym}\mathcal{T}_X}$ and the claim follows.
\end{proof}
\begin{remark}
    The set-theoretic support condition translates locally into the fact that the Higgs field is a nilpotent matrix, hence the name for the nilpotent Dolbeault stack.
\end{remark}

\section{K-Theory}
\subsection{D\'evissage}
We work over an algebraically closed field $k$ of characteristic 0. In derived algebraic geometry $K$-theory is invariant under derived nilpotent thickenings, also known as \textit{derived invariant} or as  \textit{d\'evissage}. 
\begin{theorem}[D\'evissage]
    Let $X$ be a derived scheme and let $\pi_0X$ denote its classical truncation. Then we have the following isomorphism 
    \begin{align}
        K(\mathsf{Coh}(\pi_0X)) \xlongrightarrow{\sim} K(\mathsf{Coh}(X))
    \end{align}
    Furthermore, we also have the isomorphism 
    \begin{align}
         K(\mathsf{Coh}(\pi_0X)^\heartsuit) \xlongrightarrow{\sim} K(\mathsf{Coh}(\pi_0X)).
    \end{align}
\end{theorem}
One might suspect that this will also carry over to supergeometry, which we will prove is indeed the case. \\
We give a brief informal overview of of algebraic $K$-theory. For more details see \cite{LandKTheoryNotes}. We will use the $K$-theory for exact categories, which in particular includes the $K$-theory of stable categories. The only exact categories we will be using, that are not stable, will be the abelian category of $\mathbb{Z}_2$-graded coherent sheaves. \\
To start off, let us recall how to construct algebraic K-theory using the group
completion method. If we consider the category of commutative monoids in anima
then naturally there exists a subcategory of commutative groups of anima. This
inclusion satisfies the conditions of the adjoint functor theorem, so we obtain
\begin{align}
    (-)^{\mathrm{gp}} : \mathsf{CMon}(\mathsf{Ani}) \rightleftarrows \mathsf{CGrp}(\mathsf{Ani}) : \iota.
\end{align}
Furthermore, we can identify $\mathsf{CGrp}(\mathsf{Ani}) \simeq \mathsf{Sp}_{\geq 0}$. Now a natural way to obtain a commutative monoid is by using the fact that a category is, in particular, a simplicial anima and considering the following functor
\begin{align}
\begin{split}
 \mathsf{Cat}^{\mathrm{ex},\otimes} &\longrightarrow \mathsf{CMon}(\mathsf{Cat}) \\
\mathcal{C} &\longmapsto \mathsf{Span}(\mathcal{C}, \mathsf{pr}\mathcal{C}, \mathsf{in}\mathcal{C}),   
\end{split}
\end{align}
where $\mathsf{pr}\mathcal{C}$ and $\mathsf{in}\mathcal{C}$ are two classes of maps an exact category is equipped with, where $\mathsf{Span}(\mathcal{C}, \mathsf{pr}\mathcal{C}, \mathsf{in}\mathcal{C})$ is the span category where the maps in $\mathsf{pr}\mathcal{C}$ are the left legs and $\mathsf{in}\mathcal{C}$ are the right legs. The functor $|-|$ is often called the \textit{geometric realization} and defined as 
\begin{align}
    \mathsf{Cat} \subset \mathsf{Fun}(\Delta^\mathsf{op},\mathsf{Ani}) \adjoint \mathsf{Ani}: \iota,
\end{align}
i.e. the left Kan extension of $\iota$ or equivalently the colimit over the simplicial diagram of a category $\mathcal{C}$. We remark that for $M \in \mathsf{CMon(\mathsf{Ani})}$ $M^\mathrm{gp} \simeq \Omega|M|$. Finally, we can compose our construction with the group completion and we obtain the following definitions.
\begin{definition}
    Let $\mathcal{C}$ be an exact category. The $K$-anima functor is the functor
\begin{align} 
\begin{split}
K : \mathrm{Cat}^{\mathrm{ex}} &\longrightarrow \mathrm{CGrp}(\mathrm{Ani}) \simeq \mathrm{Sp}_{\geq 0} \\
\mathcal{C} &\longmapsto \Omega|\mathrm{Span}(\mathcal{C}, \mathrm{pr}\mathcal{C}, \mathrm{in}\mathcal{C})|
\end{split}
\end{align}
\end{definition}
\begin{definition}
    The $K$-theory functor on stable categories is defined as
\begin{align} 
K : \mathrm{Cat}^{\mathrm{st}} \longrightarrow \mathrm{Cat}^{\mathrm{ex}} \longrightarrow \mathrm{Sp}_{\geq 0},
\end{align}
where $\mathrm{Cat}^{\mathrm{st}} \to \mathrm{Cat}^{\mathrm{ex}}$ sends $\mathcal{C} \mapsto \mathrm{Span}\mathcal{C}$, where the projections and injections are all
morphisms.
\end{definition}
\begin{remark}
    In particular, Quillen's exact 1-categories are also examples of exact $\infty$-categories, while being stable is a strictly $\infty$-categorical phenomenon.
\end{remark}
To prove d\'evissage let us admit the following theorem.
\begin{theorem}[{\cite[Theorem 5.48.]{LandKTheoryNotes}}] \label[theorem]{devissage}
    Let $\mathcal{A}$ be an abelian category and $\mathcal{B} \subset \mathcal{A}$ a full additive subcategory with the following properties:
    \begin{enumerate}
        \item For every exact sequence $a \hookrightarrow b \twoheadrightarrow a' \in \mathcal{A}$ with $b \in \mathcal{B}$, it follows that $a,a' \in \mathcal{B}$,
        \item every object of $a\in \mathcal{A}$ admits a finite filtration whose filtration quotient lies in $\mathcal{B}$, that is, there is a sequence 
        \begin{align*}
            0= a_0 \hookrightarrow a_1 \hookrightarrow a_2 \hookrightarrow ... \hookrightarrow a_n = a
        \end{align*}
        such that $a_{i+1}/a_i \in \mathcal{B}$ for all $i$.
    \end{enumerate}
    Then the inclusion $\mathcal{B}\rightarrow \mathcal{A}$ induces an equivalence $K(\mathcal{B})\simeq K(\mathcal{A})$.
\end{theorem}
We will now apply this to our situation of $\mathsf{sCoh}(X)$ and $\mathsf{sCoh}(X_\mathsf{bos})$.
\begin{proposition}\label[proposition]{cohcoh}
    Let $\mathfrak{X}$ be a noetherian superstack and $\mathfrak{X}_\mathsf{bos}$ its bosonic truncation then we have the following isomorphism
    \begin{align}
        K(\mathsf{sCoh}(\mathfrak{X}_\mathsf{bos})^\heartsuit) \simeq K(\mathsf{sCoh}(\mathfrak{X})^\heartsuit)
    \end{align}
\end{proposition}
\begin{proof}
    We will verify the conditions of \cref{devissage}. To verify (1), one observes that for an element $\mathcal{F}\in \mathsf{sCoh}(X_\mathsf{bos})^\heartsuit$ the action of the nilpotent ideal $\mathcal{J}$ annihilates the object, i.e. $\mathcal{J}\cdot \mathcal{F}=0$. Thus, given a short exact sequence 
    \begin{center}
        \begin{tikzcd}
            \mathcal{F}' \arrow[r] & \mathcal{F} \arrow[r] & \mathcal{F}''
        \end{tikzcd}
    \end{center}
    it follows that since $\mathcal{F}'$ is a submodule of $\mathcal{F}$ the action of $\mathcal{J}$ also annihilates $\mathcal{F}'$ and the action on the quotient module $\mathcal{F}''$ is induced by the action of $\mathcal{J}$ on $\mathcal{F}$, which proves (1). To prove (2) one notes that every object in $\mathsf{sCoh}(X)$ admits a finite filtration given by the nilpotent ideal $\mathcal{J}$, due to the noetherian assumption,
\begin{align*}
            0= \mathcal{J}^{n+1}\mathcal{F} \hookrightarrow \mathcal{J}^n\mathcal{F} \hookrightarrow \mathcal{J}^{n-1}\mathcal{F} \hookrightarrow ... \hookrightarrow \mathcal{J}^0\mathcal{F} = \mathcal{F},
        \end{align*}
        where the filtration quotients are given by $\mathcal{J}^m\mathcal{F}/\mathcal{J}^{m+1}\mathcal{F}$, which are objects of \\
        $\mathsf{sCoh}(X_\mathsf{bos})^\heartsuit$, since $\mathcal{J} \cdot \mathcal{J}^m\mathcal{F}/\mathcal{J}^{m+1}\mathcal{F} = 0$. This proves (2) and completes the verification of \cref{devissage} and the claim follows. 
\end{proof}
\begin{corollary}
    Let $X$ be a noetherian geometric superstack then the map 
    \begin{align}
        j_*:K(\mathsf{sCoh}(X_\mathsf{bos})) \longrightarrow K(\mathsf{sCoh}(X)).
    \end{align}
    is an equivalence.
\end{corollary}
\begin{proof}
    The operation of taking the heart is functorial, thus the following diagram commutes
    \begin{center}
        \begin{tikzcd}
            K(\mathsf{sCoh}(X_\mathsf{bos})) \arrow[r] & K(\mathsf{sCoh}(X))\\
            K(\mathsf{sCoh}(X_\mathsf{bos})^\heartsuit) \arrow[u] \arrow[r] & K(\mathsf{sCoh}(X)^\heartsuit) \arrow[u]
        \end{tikzcd}.
    \end{center}
    The vertical arrows are equivalences due to Barwicks theorem of the heart and the lower horizontal arrow is an equivalence by \cref{cohcoh}, thus the top horizontal arrow is also an equivalence.
\end{proof}
\subsection{Super Fundamental Classes}
Let us also complete our study of the $K$-theory spectrum by describing it in relation to the $K$-theory spectrum of ordinary modules on the bosonic reduction.
\begin{lemma} \label[lemma]{Kcommutecoproduct}
    The $K$-anima functor commutes with coproducts of $\mathsf{Cat}^\mathsf{st}$.
\end{lemma}
\begin{proof}
    We prove the binary case, where the general case follows by induction. The category of small stable categories is semi-additive hence the coproduct is given by the cartesian product. Thus, we have a Verdier sequence 
    \begin{center}
        \begin{tikzcd}
            \mathcal{C} \arrow[r] & \mathcal{C} \times \mathcal{D} \arrow[r] & \mathcal{D}
        \end{tikzcd}
    \end{center}
    There exists a functor $\iota:\mathcal{D}\to \mathcal{C}\times \mathcal{D}$, since coproducts and products are isomorphic, that is right adjoint to the projection. Thus $\mathcal{D}$ a reflective subcategory and therefore $\iota$ is a section. Using the fact that $K$-theory is localizing \cite[Theorem 5.71.]{LandKTheoryNotes}, we obtain a short exact sequence of spectra, which splits since $K(\mathsf{pr}_\mathcal{D})$ admits a section given by $K(\iota)$, thus we obtain the split exact sequence 
       \begin{center}
        \begin{tikzcd}
            K(\mathcal{C}) \arrow[r] &K( \mathcal{C}) \oplus K(\mathcal{D}) \arrow[r] & K(\mathcal{D}),
        \end{tikzcd}
    \end{center}
    as desired.
\end{proof}
To be able to compare classes of the $K$-theory of supermodules and ordinary modules, we also admit the following corollary.
\begin{corollary} \label[corollary]{formofscoh}
    Let $X$ be a noetherian superstack, then we have an isomorphism 
    \begin{align}
        K(\mathsf{sCoh}(X_\mathsf{bos})) \simeq K(\mathsf{Coh}(X_\mathsf{bos})) \times K(\mathsf{Coh}(X_\mathsf{bos}))
    \end{align}
    as spectra.
\end{corollary}
\begin{proof}
    Let us first identify $\mathsf{sCoh}(X_\mathsf{bos}) \simeq \mathsf{Fun}(\mathbb{Z}_2, \mathsf{Coh}(X_\mathsf{bos}))$. We observe that there is an adjunction
    \begin{align}
        \begin{split}
          \Delta: \mathsf{Coh}(X_\mathsf{bos}) \adjoint  \mathsf{Fun}(\mathbb{Z}_2, \mathsf{Coh}(X_\mathsf{bos})): U,
        \end{split}
    \end{align}
    where the adjunction $\Delta \dashv U$ forgets the $\mathbb{Z}_2$-grading and embeds a sheaf diagonally, i.e. $\mathcal{F} \mapsto (\mathcal{F},\Pi \mathcal{F})$. Since $\mathsf{Cat}^{\mathsf{st}}$ has all limits and colimits we may take the fibre of this functor and we obtain another Verdier sequence of the form 
    \begin{center}
        \begin{tikzcd}
            \mathsf{Coh}(X_\mathsf{bos}) \arrow[r, "\Delta"] & \mathsf{sCoh}(X_\mathsf{bos}) \arrow[r, "p"] & \mathsf{Coh}(X_\mathsf{bos}).
        \end{tikzcd}
    \end{center}
    analyzing the second map we observe that since $p$ is a quotient map that sends sheaves of the form  $\mathcal{F}= (\mathcal{F}', \Pi \mathcal{F}')$ to zero then $K(p)$ must also send them to zero, which results in
    \begin{align}
    \begin{split}
          K(p): K(\mathsf{sCoh}(X_\mathsf{bos})) &\longrightarrow K(\mathsf{Coh}(X_\mathsf{bos})) \\
       [\mathcal{F}] + [\Pi\mathcal{G}]  &\longmapsto  [\mathcal{F}] - [\mathcal{G}].
    \end{split}
    \end{align}
    One observes that $\Delta$ is fully faithful and the parity shift functor $\Pi: \mathsf{sCoh}(X_\mathsf{bos}) \to \mathsf{sCoh}(X_\mathsf{bos})$ is an involution, so by \cref{Kcommutecoproduct} $K(\mathsf{sCoh}(X_\mathsf{bos}))$ participates in a split short exact sequence of spectra of the form
    \begin{center}
        \begin{tikzcd}[column sep = large]
            K(\mathsf{Coh}(X_\mathsf{bos})) \arrow[r, "K(\Delta)"] \arrow[d, equal] & \arrow[l, bend right, "K(U)"'] K(\mathsf{sCoh}(X_\mathsf{bos})) \arrow[d, "\wr"] \arrow[r, "K(p)"] & K(\mathsf{Coh}(X_\mathsf{bos})) \arrow[d, equal] \\
            K(\mathsf{Coh}(X_\mathsf{bos})) \arrow[r, "K(\Delta)"] &  K(\mathsf{Coh}(X_\mathsf{bos}))[\sigma]/(\sigma^2-1) \arrow[r, "K(p)"] & K(\mathsf{Coh}(X_\mathsf{bos})),
        \end{tikzcd}
    \end{center}
    where $\sigma =[\Pi]$ is a formal square root of $1$ and represents the parity shift functor, which proves the claim.
\end{proof}
\begin{remark}
One now observes that the morphisms
     \begin{align}
       K(U),  K(p): K(\mathsf{sCoh}(X_\mathsf{bos})) &\longrightarrow K(\mathsf{Coh}(X_\mathsf{bos}))
    \end{align}
     are the categorical incarnation of the specialization morphism of setting $\sigma = \pm 1$. For our purposes, it is more useful to use the specialization map $\sigma=-1$.
\end{remark}
We are now in a position to give a definition for a super analogue of the virtual fundamental class.
\begin{definition}[Super Fundamental $K$-Class]
    Let $X$ be an Artin superstack and $\iota: X_\mathsf{bos}\rightarrow X$ be the inclusion of the bosonic truncation. We define the \textit{super fundamental class} $ [X]^K$ to be the unique class that satisfies
    \begin{align}
        \iota_*[X]^K = [\mathcal{O}_X] \in K(\mathsf{sCoh}(X_\mathsf{bos}))
    \end{align}
\end{definition}
By considering the composition 
\begin{align}
    K(\mathsf{sCoh}(X)) \xlongrightarrow{\sim} K(\mathsf{sCoh}(X_\mathsf{bos})) \xlongrightarrow{K(p)} K(\mathsf{Coh}(X_\mathsf{bos})),
\end{align}
we may define the following:
\begin{definition}[Super Fundamental Class]
    Let $X$ be an Artin superstack and $\iota: X_\mathsf{bos}\rightarrow X$ be the inclusion of the bosonic truncation and, by abuse of notation, we denote the image of $[\mathcal{O}_X]$ in $K(\mathsf{Coh}(X_\mathsf{bos}))$ by $[\mathcal{O}_X]$. Then the \textit{super fundamental class} $ [X]$, in the Chow ring, is defined to be 
    \begin{align}
        [X] \coloneqq \tau_X([\mathcal{O}_X]) \cap \mathsf{Td}([\mathcal{T}_{X}])^{-1}
    \end{align}
\end{definition}
\begin{proposition} \label[proposition]{fundamentalclass}
    Let $X$ be a smooth Deligne-Mumford superstack then 
    \begin{align}
        [X] \simeq c_r(\mathcal{E}^\vee) \cdot [X] 
    \end{align}
\end{proposition}
\begin{proof}
Let $\mathcal{J}$ be the ideal generated by all odd nilpotents and denote $\mathcal{E}\coloneqq \mathcal{J}/\mathcal{J}^2$. One observes that $\mathcal{O}_X$ also fits into a short exact sequence of the form
\begin{center}
    \begin{tikzcd}
        \mathcal{J}/\mathcal{J}^2 \arrow[r] & \mathcal{O}_X \arrow[r] & \mathcal{O}_{X_\mathsf{bos}},
    \end{tikzcd}
\end{center}
which together with the filtration sequence induced by the nilpotent ideal $\mathcal{J}$ one computes an alternative description of $\mathcal{O}_X$ as follows:
\begin{align}
\begin{split}
      [\mathcal{O}_X] &\simeq [\mathcal{J}^1] + [\wedge^0\mathcal{E}] \\
      &\simeq [\mathcal{J}^2] - [\wedge^1 \mathcal{E}] + [\wedge^0\mathcal{E}] \\
      & \ \ \vdots \\
      & \simeq \Sigma_i(-1)^i[\wedge^i\mathcal{E}].
\end{split}
\end{align}
The extra minus signs appear, because we used the specalization map $\sigma = -1$ and the fact that $\wedge^{2k+1}\mathcal{E}$ are vector bundles in degree 1. We can now also use the identity 
\begin{align}
    \sum_{p=0}^r (-1^p)\mathsf{ch}(\wedge^p \mathcal{E}^\vee) = c_r(\mathcal{E}) \cdot \mathsf{Td}(\mathcal{E})^{-1},
\end{align}
which is \cite[Example 3.2.5]{fulton2013intersection}, to identify the Chern character of $\mathcal{O}_X$ as
\begin{align}
 \mathsf{ch}(\mathcal{O}_X) =   \sum_{p=0}^r (-1^p)\mathsf{ch}(\wedge^p \mathcal{E}^\vee) = c_r(\mathcal{E}) \cdot \mathsf{Td}(\mathcal{E})^{-1}.
\end{align}
Furthermore, $\mathcal{T}_X$ fits into a short exact sequence 
\begin{center}
    \begin{tikzcd}
        \mathcal{T}_{X_\mathsf{bos}} \arrow[r] & \mathcal{T}_{X}|_{X_\mathsf{bos}} \arrow[r] & (\mathcal{J}/\mathcal{J}^2)^\vee,
    \end{tikzcd}
\end{center}
which lets us identify $[\mathcal{T}_X] \simeq [\mathcal{T}_{X_\mathsf{bos}}]-[\mathcal{E}^\vee] \in K(\mathsf{Coh}(X_\mathsf{bos}))$ and one now computes

    \begin{align}
    \begin{split}
          [X] &\coloneqq \tau_{X_\mathsf{bos}}([\mathcal{O}_X]) \cap \mathsf{Td}([\mathcal{T}_{X}])^{-1} \\
        &\simeq \mathsf{Td}([\mathcal{T}_{X_\mathsf{bos}}]-[\mathcal{E}]^\vee)^{-1} \cap (\tau_{X_\mathsf{bos}}([\mathcal{O}_X]) \\
        &\simeq \mathsf{Td}([\mathcal{T}_{X_\mathsf{bos}}])^{-1} \cdot \mathsf{Td}([\mathcal{E}^\vee]) \cap (\mathsf{ch}([\mathcal{O}_X]) \cdot \mathsf{Td}(\mathcal{T}_{X_\mathsf{bos}}) \cap [X_\mathsf{bos}]))\\
        &\simeq \mathsf{Td}([\mathcal{E}^\vee]) \cap c_r(\mathcal{E}^\vee) \mathsf{Td}(\mathcal{E}^\vee)^{-1} \cap [X_\mathsf{bos}] \\
        &\simeq c_r(\mathcal{E}^\vee) \cap [X_\mathsf{bos}],
    \end{split}
    \end{align}
    as desired.
\end{proof}
As an application we will compute the super fundamental class of the moduli space of super Riemann surfaces and show they agree with $\Theta$-classes defined in \cite{norbury2023new}.
\begin{corollary}
    Let $\overline{\mathfrak{M}}_{g,n}$ be the Deligne-Mumford compactification of the  moduli of genus $g$ super Riemann surfaces with $n$ Neveu-Schwarz punctures. Then the super fundamental class is given by 
    \begin{align}
        [\overline{\mathfrak{M}}_{g,n}] = c_{\operatorname{top}}(\mathcal{E}^\vee) \cap [\overline{\mathcal{SM}}_{g,n}] \in A_*(\overline{\mathcal{SM}},\mathbb{Q}).
    \end{align}
\end{corollary}
\begin{proof}
  The underlying bosonic space of the moduli of super Riemann surfaces is given by the moduli of stable twisted spin curves 
  \begin{align}
      (\overline{\mathfrak{M}}_{g,n})_\mathsf{bos} \simeq \overline{\mathcal{SM}}_{g,n}
  \end{align}
  The tangent bundle of $\overline{\mathfrak{M}}_{g,n}$ splits along the bosonic locus into 
  \begin{align}
      \mathcal{T}_{\mathcal{SM}_{g,n}} \oplus -R\pi_*\mathcal{E}^\vee, 
  \end{align}
  and we will define $\widehat{\mathcal{E}}_{g,n} \coloneqq -R\pi_*\mathcal{E}^\vee$.
  Thus, by \cref{fundamentalclass} we can identify the super fundamental class to be 
  \begin{align}
      [\overline{\mathfrak{M}}_{g,n}] \simeq c_\mathrm{top}(\widehat{\mathcal{E}}_{g,n}) \cap [\overline{\mathcal{SM}}_{g,n}],
  \end{align}
  as wished.
\end{proof}
To connect this class to the $\Theta$-classes, we admit the following proposition:
\begin{proposition}
Let $p: \overline{\mathcal{SM}}_{g,n}\longrightarrow \overline{\mathcal{M}}_{g,n}$ the forgetful morphism, then we have the following identity:
\begin{align}
	p_*[\overline{\mathfrak{M}}_{g,n}] = \Theta_{g,n} \cap [\overline{\mathcal{M}}_{g,n}].
\end{align}
\end{proposition}
\begin{proof}
	The morphism $p: \overline{\mathcal{SM}}_{g,n} \rightarrow \overline{\mathcal{M}}_{g,n}$ is proper and flat \cite[Corollary 2.2.2.]{abramovich2003moduli}, thus it admits both $p^*$ and $p_*$ on Chow groups \cite[Proposition 3.7.]{vistoli1989intersection}, and we also recall that for a flat morphism of stacks $f: X \rightarrow Y$, we have
	\begin{align}
    \begin{split}
      		f^*: A_*(Y) &\rightarrow A_*(X) \\
 		[Y]		&\mapsto f^*[Y] = [X],  
    \end{split}
	\end{align}
	which lets us identify $p^*[\overline{\mathcal{M}}_{g,n}]=[\overline{\mathcal{SM}}_{g,n}]$ and observes that
	\begin{align}
		[\mathfrak{M}_{g,n}] = c_{\operatorname{top}}(\widehat{\mathcal{E}}_{g,n}) \cap p^*[\overline{\mathcal{M}}_{g,n}].
	\end{align}
	    By applying the proper pushforward we obtain, by the projection formula, 
	\begin{align}
		p_*[\overline{\mathfrak{M}}_{g,n}] = \Theta_{g,n} \cap [\overline{\mathcal{M}}_{g,n}],
	\end{align}
	where $\Theta_{g,n} = p_*c_{\operatorname{top}}(\widehat{\mathcal{E}}_{g,n})$.
\end{proof}
\begin{remark}
    This agrees with the $\Theta$-class in \cite{norbury2023new}, up to a factor of $(-1)^n2^{g-1+n}$, which is inessential, as it was chosen to simplify the gluing relations of the $\Theta$-classes \cite{norbury2020enumerative}.
\end{remark}

\printbibliography
\end{document}